\documentclass{amsart}

\usepackage{amsmath,amssymb,enumerate, amsfonts}

\usepackage{epsfig}
\usepackage{color}

\newtheorem{theorem}{Theorem}[section]

\newtheorem{addendum}[theorem]{Addendum}

\newtheorem{proposition}[theorem]{Proposition}
\newtheorem{lemma}[theorem]{Lemma}
\newtheorem{corollary}[theorem]{Corollary}

\newtheorem{definition}[theorem]{Definition}
\newtheorem{remark}[theorem]{Remark}
\newtheorem{remarks}[theorem]{Remarks}
\newtheorem{claim}[theorem]{Claim}
\newtheorem{question}{Question}

\newcommand{\new}[1]{{\bf#1}}

\newcommand\conorm{\operatorname{conorm}}

\newcommand\Diff{\operatorname{Diff}}
\newcommand\eps{\epsilon}

\newcommand\Cyl{\mathcal C}

\newcommand\cF{\mathcal F}

\newcommand\cT{\mathcal T}
\newcommand\Fs{\mathcal F^s}
\newcommand\Fc{\mathcal F^{c}}

\newcommand\Fcs{\mathcal F^{cs}}
\newcommand\Fcu{\mathcal F^{cu}}
\newcommand\Fu{\mathcal F^{u}}

\newcommand\LL{{\mathbb L}}

\newcommand\Lcu{\lambda^{cu}}

\newcommand\loc{{\operatorname{loc}}}

\newcommand\nucu{m^{cu}}

\newcommand\nuu{m^{u}}
\newcommand\NN{{\mathbb N}}

\newcommand\Proberg{{\mathbb P}_{\operatorname{erg}}}
\newcommand\RR{{\mathbb R}}
\newcommand\supp{\operatorname{supp}}
\renewcommand\top{{\operatorname{top}}}

\newcommand\mftPH{\rm(I)}
\newcommand\mftDC{\rm(II)}
\newcommand\mftCLeaf{\rm(III)}
\newcommand\mftPush{\rm(IV)}

\begin{document}

\title[Dichotomy for measures of maximal entropy]{A dichotomy for measures of maximal entropy
near time-one maps of transitive Anosov flows
}

\begin{abstract}
We show that time-one maps of transitive Anosov flows of compact manifolds are accumulated by diffeomorphisms robustly satisfying the following dichotomy: either all of the measures of maximal entropy are non-hyperbolic, or there are exactly two ergodic measures of maximal entropy, one with a positive central exponent and the other with a negative central exponent.

We establish this dichotomy for certain partially hyperbolic diffeomorphisms isotopic to the identity whenever both of their strong foliations are minimal. Our proof builds on the approach developed by Margulis for Anosov flows where he constructs suitable families of measures on the dynamical foliations.

\medbreak
\noindent
{\sc R\'esum\'e.} Nous montrons que l'application du temps $1$ de tout flot d'Anosov transitif est dans l'adh\'erence des diff\'eomorphismes pr\'esentant robustement la dichotomie suivante: ou bien aucune mesure maximisant l'entropie n'est hyperbolique, ou bien il existe exactement deux mesures ergodiques maximisant l'entropie, l'une ayant un exposant central strictement positif, l'autre strictement n\'egatif.

Nous \'etablissons cette dichotomie pour certains diff\'eomorphismes partiellement hyperboliques isotopes \`a l'identit\'e sous l'hypoth\`ese de la minimalit\'e de leurs feuilletages invariants forts. Notre preuve s'appuie sur l'approche d\'evelopp\'ee par Margulis dans le cas des flots d'Anosov et la construction de familles de mesures convenables sur les diff\'erents feuilletages dynamiques.
\end{abstract}

\author{J\'er\^ome Buzzi, Todd Fisher, and Ali Tahzibi}
\address{J.~Buzzi, Laboratoire de Mathématiques d'Orsay, CNRS - UMR 8628
Universit\'e Paris-Saclay, 91405 Orsay, France, \emph{E-mail address:}
\tt{jerome.buzzi@universite-paris-saclay.fr}}
\address{T.~Fisher, Department of Mathematics, Brigham Young University, Provo, UT 84602, \emph{E-mail address:} \tt{tfisher@mathematics.byu.edu}}
\address{A.~Tahzibi, Instituto de Ci\^{e}ncias Matem\'aticas e de Computa\c{c}\~ao, Universidade de S\~ao Paulo (USP), \emph{E-mail address:} \tt{tahzibi@icmc.usp.br}}

\thanks{ T.F.\ is supported by Simons Foundation grant \# 239708. A.T is supported by FAPESP 107/06463-3 and CNPq (PQ) 303025/2015-8. A.T.'s stay at the Institut de Math\'ematique d'Orsay was partially supported by the ANR project ISDEEC \# ANR-16-CE40-0013.}


\maketitle

\section{Introduction}

{In his pioneering work \cite{Margulis70}, Margulis studied measures of maximal entropy of geodesic flows in order to count closed geodesics for manifolds with variable negative curvature. More precisely, }
he constructed  {a family of} measures $\{m_x\}_{x\in M}$ such that for all $x\in M$ the measure $m_x$ is carried by 
the unstable manifold at $x$, and for all $t\in\RR$ we have
 $$ (\varphi^t)_* m_x=e^{-t \cdot h_\top(\varphi)}m_{\varphi^t x}.$$  
 He then
{built an invariant probability measure which was observed to be a measure of maximal entropy and is now} called the Bowen-Margulis measure. {It was then proved 
to be the unique measure of maximal entropy. We refer to Ledrappier \cite{Ledrappier13} for an introduction.}

 In this paper, we will extend Margulis' construction to a class of partially hyperbolic maps  and obtain a striking dichotomy.

\begin{theorem}
\label{mainthm0}
If $\varphi^t$ is a  transitive Anosov flow on a compact manifold $M$, then there is an open set $\mathcal U$ in $\Diff^1(M)$ which contains $\varphi^1$ in its closure
 such that for any $f\in \mathcal{U} \cap \mathrm{Diff}^2(M)$ we have the following dichotomy:
\begin{enumerate}
\item either all the measures of maximal entropy have zero central Lyapunov exponents, or
\item there are exactly two ergodic measures of maximal entropy where one has a positive central exponent and {the other} has a negative central exponent, and both measures are Bernoulli.
\end{enumerate}
\end{theorem}

\subsection*{Related results}

These results are part of a larger program to understand properties of entropy beyond uniform hyperbolicity.  In that classical setting, say for a transitive Anosov diffeomorphism, there is a unique measure of maximal entropy.\footnote{{Throughout this paper, we will abbreviate \emph{ergodic Borel probability measure maximizing the entropy} to MME. Recall that a measure maximizes the entropy if and only if its is invariant and almost all its ergodic components are MMEs.}} {Beyond the hyperbolic setting,} even though there are a number of significant results \cite{Newhouse89, BCS,BF} there are still many fundamental open questions.  Partially hyperbolic diffeomorphisms with one-dimensional center have been studied as the ``next nontrivial class". 
A MME always exists in this setting by entropy expansivity (see  \cite{CowiesonYoung,DFPV, LVY}). Its uniqueness is a more delicate question.

{This uniqueness} has been shown for  certain systems that are \emph{derived from Anosov}, a subclass introduced by Ma\~n\'e, first for specific constructions, then in greater and greater generality \cite{BFSV,CT,U,FPS,BCF}.

The partially hyperbolic diffeomorphisms with a center foliation into circles form another subclass with a more subtle behavior. Assuming accessibility, \cite{RHRHTU} has established the following dichotomy:
 \begin{itemize}
  \item[--] either the dynamics is isometric in the center direction  and there exists a unique MME which is nonhyperbolic; or
  \item[--] there are multiple hyperbolic MMEs.
\end{itemize}

\subsection*{Strategy of proof}

We introduce a new subclass of partially hyperbolic diffeomorphisms with one-dimensional center which we call \emph{flow-type}. They are isotopic to the identity and the fundamental examples are the perturbations of time-one maps of Anosov flows.
Our main result is {Theorem~\ref{mainthm-dichotomy}: the above dichotomy holds for partially hyperbolic flow-type diffeomorphisms whose strong foliations are both minimal.
  
The uniqueness of the MME for a given sign  of the central exponent {(say nonpositive)} follows from a variant of Margulis' approach. Namely, we first build a family of measures on the center-unstable leaves {which is invariant under stable holonomies and projectively invariant under the dynamics}. Then we construct measures on unstable leaves, which we call Margulis $u$-conditionals. This is more difficult for maps than for flows.

We then use the entropy with respect to the unstable foliation as introduced by Ledrappier and Young \cite{LYII} and an argument of Ledrappier \cite{L84} to show that, when its central exponent is nonpositive, a measure maximizes the entropy if and only if its disintegration along the unstable leaves {is given by} the Margulis conditionals.

A Hopf argument shows that if there is a MME with negative central exponent, then any MME with nonpositive central exponent must coincide with it.
The symmetry between positive and negative central exponents in the hyperbolic case follows from the one-dimensionality of the central leaves: we associate to any measure with, say negative central exponent, an isomorphic one with nonnegative central exponent.

A hyperbolic MME is isomorphic to a Bernoulli shift times a circular permutation, according to a general result by Ben Ovadia \cite{Benovadia}. The triviality of the permutation follows by considering iterates.
This concludes the proof of Theorem~\ref{mainthm-dichotomy}.

\medbreak

Finally, to prove  Theorem~\ref{mainthm0} we establish Theorem~\ref{mainthm-mft}, i.e., we  find open sets of partially hyperbolic flow-type diffeomorphisms with both strong foliations minimal near any time-one map of a transitive Anosov flow. We first show that such time-one maps are robustly  flow-type, partially hyperbolic diffeomorphisms. Then  Bonatti and D\'iaz~\cite{BD96} provide a perturbation ensuring robust transitivity. Lastly, by a further perturbation {following} Bonatti, D\'iaz, and Ur\`es~\cite{BDU}  we ensure robust minimality of both strong foliations.   Theorem~\ref{mainthm0}  now follows from Theorem~\ref{mainthm-dichotomy}. 

\subsection*{The use of Margulis conditionals}
The construction of Margulis has given rise to a large body of work, mainly devoted to the estimation of the number of periodic orbits, sometimes beyond the uniformly hyperbolic setting \cite{Knieper97}. We refer to Sharp's survey  in  \cite{Margulis2004}, the long awaited publication of Margulis' thesis. The works of Hamenst\"{a}dt \cite{Ha89} and Hasselblatt \cite{H89} that give a geometric description of the Margulis conditionals $\{m_x\}_{x\in M}$ are perhaps closer to our concerns. {We also note Plante's point of view \cite{Plante75} (see also \cite{RuelleSullivan75,Ruelle75}): measures invariant under stable holonomies can be seen as transverse measures. We have not pursued this point of view here.} 

While this work was being written, we learned that a different but related approach has been developed in \cite{CPZ}. This approach can deal with equilibrium measures (i.e., generalizations of measures of maximal entropy taking into account a weight function) but requires non-expansion along the center. Separately, Jiagang Yang has told us that he  also {has} {some related results.} 

\subsection*{Comments}

{
First, we note that our result, Theorem~\ref{mainthm-dichotomy}, will be proved for a class of partially hyperbolic diffeomorphisms which we call \emph{flow-type} (see Def.~\ref{flowtype}). This class turns out to coincide\footnote{Discretized Anosov flows are partially hyperbolic diffeomorphisms that can be written as $x\mapsto\psi^{\tau(x)}(x)$ for some topological Anosov flow $\psi$ and some positive continuous function~$\tau$. In dimension~3, \cite[Proposition G.2]{BFFP1} shows that discretized Anosov flows are, among the partially hyperbolic diffeomorphisms, exactly those which satisfy our Definition~\ref{flowtype} with the additional condition that the underlying flow~$\psi$ is topological Anosov. However, this is automatic by \cite[Theorem 2]{BW}.},
at least in dimension~3, with the \emph{discretized Anosov flows} introduced by Barthelm\'e, Fenley, Frankel, and Potrie in the classification of partially hyperbolic diffeomorphisms \cite{BFFP1}.  Work by Martinchich \cite{Martinchich} is extending this to any dimension, and proves that the set of such diffeomorphisms is both open and closed among partially hyperbolic diffeomorphisms (making this property invariant by isotopy). 
}

{
We point out that even if our result main's application is to perturbations of time-one maps of flows with very good properties, those perturbations are almost never time-one maps of a  perturbed flow \cite{Palis74}.
This well-known fact explains why some properties, such as stable ergodicity  \cite{GPS94}, have been extended from time-one maps to their perturbations only with great difficulties. 
}

Second, we note that part of our results could be obtained from symbolic dynamics, using generalizations of ideas going back to  the classical works of Sinai, Ruelle, and Bowen (see, e.g., \cite{Sinai72,BowenBook,RuelleBook}). More precisely, existence of at most one MME with, say, positive central exponent can be deduced from \cite[Section 1.6]{BCS} since, in the terminology of this work, our minimality assumption implies that there is a unique homoclinic class of measures with a given sign of the central exponent. However, the dichotomy does not seem to follow from this approach {which is blind to nonhyperbolic measures.}

Thirdly, one usually expects that results such as ours can be extended to $C^{1+\alpha}$ smoothness, for any $0<\alpha<1$, and generalized to equilibrium measures with respect to H\"older-continuous potentials (although {uniqueness holds} for generic potentials \cite{RuelleBook}).

\subsection*{Questions}
{
Our techniques demand a very strong form of irreducibility: the minimality of both strong foliations. We ask:
}
\begin{question}
In Theorem~\ref{mainthm-dichotomy}, can one replace minimality of both strong foliations by minimality of just one {of those foliations} or by robust transitivity {of the diffeomorphism}? 
\end{question}

{This minimality is obtained in Theorem \ref{mainthm0} by a $C^1$ perturbation assuming the topological transitivity of the Anosov flow.  We believe that some non transitive Anosov flows could be accumulated by diffeomorphisms having more than two MMEs.}

{
We expect both the hyperbolic and nonhyperbolic cases to occur with the hyperbolic case being generic. Indeed,} in the conservative  setting, there is a {very strong} rigidity result \cite{AVW}. We think that some version of it may hold for {MMEs in the dissipative setting.}

\begin{question}\label{q.dichotomy}
In the setting of Theorem~\ref{mainthm-dichotomy}, is the hyperbolic case open and dense? When the MME is nonhyperbolic, does this imply that the diffeomorphism is the time one map of a flow? does it at least exclude the existence of hyperbolic periodic points?
\end{question}

Though we will identify the disintegrations of nonhyperbolic MMEs along \emph{both} strong foliations, {our} analysis remains incomplete:

\begin{question}
Consider a partially hyperbolic diffeomorphism $f$ with flow-type and with minimality of both strong foliations. {Assume that its MMEs are nonhyperbolic.} Can their disintegration along the center be atomic like in the hyperbolic case? Can there be more than one nonhyperbolic MME? Are nonhyperbolic MMEs Bernoulli?
\end{question}

{
\begin{remark}
After the preprint version of this work was released, Crovisier and Poletti \cite{CrovisierPoletti} {have announced} positive answers to Questions 2 and 3. They generalized Avila-Viana's invariance principle  {(see  \cite{L}, \cite{AV} and \cite{TY})} to partial hyperbolicity with quasi-isometric center  and applied it to the nonhyperbolic MMEs in our setting using the disintegration we establish here (see Proposition~\ref{prop-Ledrappier}).
\end{remark}
}

We prove that the hyperbolic MMEs are  {Bernoulli, hence strongly mixing. One can try to establish some speed (see \cite{YangYang} for a related result).}

\begin{question}
If $\mu$ is a hyperbolic MME for a flow-type diffeomorphism $f$ with minimality of both strong foliations, does it satisfy \emph{exponential decay of correlations} for H\"older-continuous functions, i.e., for any H\"older-continuous functions $u,v:M\to\RR$, does there exist a number $\kappa<1$ such that:
 $$
    \int_M u\circ f^n . v \, d\mu - \int_M u \, d\mu \int_M  v \, d\mu = O(\kappa^n)?
 $$
\end{question}

For Anosov flows, the topological entropy can obviously be changed by perturbations whereas it is locally constant for Anosov diffeomorphisms. What is the situation for the maps we consider?

\begin{question}
Consider a flow-type diffeomorphism whose strong foliations are robustly both minimal.
Is it true that the volume growth of each strong leaf is equal to the topological entropy?
Does the topological entropy have a homological interpretation? {Can an arbitrarily small perturbation make the topological entropy locally constant as a function of the diffeomorphism?}
\end{question}

Our approach avoids constructing the Margulis conditional on strong leaves that meet a compact leaf. We do not know if the construction could be extended to the whole manifold, e.g., by passing to some cover. Such continuity would imply uniqueness of these Margulis systems at least in the hyperbolic case. We therefore ask:

\begin{question}
Does a flow-type diffeomorphism with both strong foliation robustly minimal admit continuous Margulis conditional on those? 
\end{question}

\subsection*{Acknowledgments}
We are grateful to Sylvain Crovisier for stimulating discussions {as well as to Rafael Potrie and Santiago Martinchich for explaining some of their results to us. We also thank the referees for helping us improve this paper.} T.F. and A.T. thank the Laboratoire de Math\'ematiques d'Orsay (C.N.R.S. and Universit\'e Paris-Sud) for its hospitality. 


\section{Background}\label{s-back}


In this section we review concepts of partial hyperbolicity, Lyapunov exponents, and disintegration of measures. 

\subsection{Partial hyperbolicity}\label{secPH}

For a diffeomorphism $f:M\to M$ of a compact manifold to itself recall the norm and conorm with respect to a subspace of $V\subset T_xM$ for some $x\in M$: $\|Df| V\|:=\max\{\|Tf(v)\|:v\in V,\; \|v\|=1\}$ and
 $$\begin{aligned}
   &\conorm(Df| V) := \min\{\|Tf(v)\|:v\in V,\; \|v\|=1\}.
 \end{aligned}$$
 A splitting $E\oplus F$ is dominated\footnote{This is sometimes called \emph{pointwise} domination, see \cite{AbdenurViana}.} if it is nontrivial, invariant, and if there is some $N\geq1$ such that, for all $x\in M$:
   $$
      \|Df^N|E_x\|<\frac12\conorm(Df^N|F_x).
   $$

\begin{definition}
A diffeomorphism is  \new{(strongly) partially hyperbolic} if there is an invariant splitting of the tangent bundle: $TM=E^s\oplus E^c\oplus E^u$ such that $E^s\oplus(E^c\oplus E^u)$ and $(E^s\oplus E^c)\oplus E^u$ are dominated, $E^s$ is uniformly contracted, and $E^u$ is uniformly expanded. 
\end{definition}

The stable and unstable bundles $E^s$ and $E^u$ of a  partially hyperbolic diffeomorphism are always uniquely integrable into stable and unstable foliations,
respectively, denoted by $\Fs$ and $\Fu$.  The bundles $E^c$, $E^{cs}:=E^s\oplus E^c$, and $E^{cu}:=E^c\oplus E^u$ fail to be integrable for some strongly partially hyperbolic diffeomorphisms.

\begin{definition}
A strongly partially hyperbolic diffeomorphism $f$ is \new{dynamically coherent} if there exists {invariant} foliations $\Fcs$ and $\Fcu$ that are tangent to the $E^{cs}$ and $E^{cu}$ bundles respectively.  In this case there is a center foliation $\Fc$ given by $\Fc(x)=\Fcs(x)\cap\Fcu(x)$ for $x\in M$. 
\end{definition}

 {We refer to} \cite{BW} for various other definitions of dynamical coherence and their  relationships.

For  a dynamically coherent diffeomorphism $f$ each leaf of $\Fcs$ is subfoliated by  the leaves of $\Fc$ and the leaves of $\Fs$.  A similar statement holds for the center-unstable foliation.  Then for any points $p, q\in M$ where $q\in\Fs(p)$ there is a neighborhood $U_p$ of $p$ in the leaf $\Fc(p)$ and a homeomorphism $h^s_{p,q}:U_p\to \Fc(q)$ {which slides points along the stable manifolds inside a foliation chart, i.e.,}
$$
h^s_{p,q}(x)\in \Fs_{{\loc}}(x)\cap \Fc_{\mathrm{loc}}(q).
$$
The map $h^s_{p,q}$ is the \textit{local stable holonomy map}. {A holonomy map  is any homeomorphism $h:A\to B$ which coincides locally with a composition of such local holonomy maps. Its \emph{size} is $\Delta:=\sup_{x\in A} d^s(p,h(p))$ where $d^s$ is the distance given by the intrinsic metric of the stable manifolds. We then say that $A$ and $B$ are \emph{$\Delta$-equivalent along $\Fs$} or $(s,\Delta)$-equivalent for short. These notions extends in the obvious way to the other dynamical foliations.}  Unstable holonomy maps are defined similarly.}

\subsection{Center Lyapunov exponents}

For a strongly partially hyperbolic diffeomorphism $f:M\to M$ a real number $\chi$ is a \textit{center Lyapunov exponent} at $x\in M$ if there exists a nonzero vector $v\in E^c_x$ such that
$$
\lim_{n\to \infty} \frac{1}{n} \log \|Df^n(v)\|=\chi.
$$
If $\mathrm{dim} E^c=1$, then the limit above only depends on $x$ and exists $m$-almost everywhere for any $f$-invariant Borel probability measure $m$.  For $m$ an ergodic $f$-invariant Borel probability measure the limit takes on a single value for $m$-almost every $x\in M$.

\subsection{Disintegration of a measure}\label{s-disint}

Let $X$ be a Polish  space and $\mu$ be a finite Borel measure on $X$.  Let $\mathcal{P}$ be a partition of $X$ into measurable sets.  Let $\hat{\mu}$ be the induced measure on the $\sigma$-algebra generated by $\mathcal{P}$.  A \textit{system of conditional measures of $\mu$} with respect to $\mathcal{P}$ is a family $\{\mu_P\}_{P\in\mathcal{P}}$ of probability measures on $X$ such that 
\begin{enumerate}
\item $\mu_P(P)=1$ for $\mu$-almost every $P\in \mathcal{P}$, and
\item given any continuous function $\psi:X\to \mathbb{R}$, the function $P\mapsto \int\psi d\mu_P$ is integrable, and 
$$
\int_X \psi d\mu= \int_\mathcal{P}\left( \int \psi d\mu_P \right) d\hat{\mu}(P).
$$
\end{enumerate}
Rokhlin \cite{Rok49, Rok67} proved that if $\mathcal{P}$ is a measurable partition, then the disintegration always exists and is essentially unique.

We will consider partitions given by foliations of a manifold.  If a foliation has a positive measure set of noncompact leaves, then the result of Rokhlin does not immediately apply. { However,  one can extend the result of Rokhlin by disintegrating  into \textit{measures defined up to scaling}  (see Avila, Viana, and Wilkinson \cite{AVW}).}

Let $M$ be a manifold where $\mathrm{dim}(M)\geq 2$ and $m$ be  a locally finite measure on $M$.  Let $\mathcal{B}$ be a small foliation box.  Then Rokhlin's result implies there is a disintegration $\{m_x^\mathcal{B}\, :\, x\in\mathcal{B}\}$ of the restriction of $m$ to the foliation box into conditional probability measures along the local leaves of the foliation, {i.e., the connected components $\mathcal F_{\mathcal B}(x)$ of $\mathcal F(x)\cap\mathcal B$ for $x\in\mathcal B$.}  From \cite[Lemma 3.2]{AVW} we know that if $\mathcal{B}$ and $\mathcal{B}'$ are foliation boxes and $m$-almost any $x\in \mathcal{B}\cap \mathcal{B}'$, then the restriction of  $m_x^\mathcal{B}$ and $m_x^{\mathcal{B}'}$ coincide up to a constant factor.

We then {know} that for $m$-almost every $x\in M$ there is a {projective} measure $m_x$ (i.e., defined up to some scaling {possibly depending on $x$}) such that $m_x(M\setminus \mathcal{F}(x))=0$. 
Furthermore, the function $x\mapsto m_x$ is constant along the leaves of $\mathcal{F}$, and the conditional probabilities $m_x^\mathcal{B}$ along the local leaves of any foliation box $\mathcal{B}$ coincide almost everywhere with the normalized restriction of the $m_x$ to the local leaves of $\mathcal{B}$.

{Finally, we note that if the measure {$m$} is invariant and the one-dimensional foliation $\cF$ is fixed by some diffeomorphism (i.e., $f(\cF(x))=\cF(x)$) without fixed points, one can replace the projective measures by true measures using the global normalization: $m_x([x,f(x))_c)=1$ for all $x\in M$ where $[x, f(x))_c)$ is the segment in the center manifold between the points $x$ and $f(x)$.}

\subsection{Measurable systems of measures}
{We will work with families of measures carried by the leaves of the dynamical foliations up to a union of exceptional leaves.
}

\begin{definition}\label{def-sys-meas}
Given a foliation $\cF$ of some manifold $M$ and {some $\cF$-saturated subset $M_1\subset M$} a \new{measurable system of measures} on $\cF|_{M_1}$ is a family $\{m_x\}_{x\in M_1}$ such that:
 \begin{enumerate}[(i)]
  \item for all $x\in M_1$, $m_x$ is a Radon measure on $\cF(x)$;
  \item for all $x,y\in M_1$, $m_x=m_y$ if $\cF(x)=\cF(y)$;
  \item $M$ is covered by foliation charts $B$ such that: $x\mapsto m_x(\phi|\cF_B(x))$ is {measurable} {on $M_1$} for any $\phi\in C_c(B)$.
\end{enumerate}
\end{definition}

The Radon property (i) means that each $m_x$ is a Borel measure and is finite on compact subsets of the leaf $\cF(x)$ (here, and elsewhere, we consider the intrinsic topology {on each leaf}).

{If $\{\mu_x\}_{x\in M}$ is the disintegration of some probability measure $\mu$ along a foliation $\cF$ as defined in the previous definition and if $\{m_x\}_{x\in M_1}$ is a {measurable} system of measures on $\cF|_{M_1}$, we will say that they \emph{coincide} if $\mu(M_1)=1$ and for $\mu$-a.e. $x\in M_1$, $\mu_x$ and $m_x$ are proportional.
}

\begin{definition}\label{defMargulisSys}
{Assume that $\cF$ is a foliation which is {invariant} under some diffeomorphism $f:M\to M$, i.e., for all $x\in M$, $f(\cF(x))=\cF(f(x))$. Let $M_1\subset M$ be $\cF$-saturated.}
A {measurable} system of measures $\{m_x\}_{x\in M_1}$ on $\cF|_{M_1}$ is \new{dilated} if there is some number  {$D > 0$} such that for all {$x\in M_1\cap f^{-1}(M_1)$}:
  \begin{equation}\label{eq-dilation}
 { f_*m_x=D^{-1}m_{f (x).} }
  \end{equation}
 {$D$} is called the \new{dilation factor}. We call {the family $\{m_x\}_{x\in M_1}$ a \new{Margulis system} on $\cF$} and the measures $m_x$ the \new{Margulis $\cF$-conditionals}.
\end{definition} 

Our construction (following Margulis) relies on {properties of } 
{the }holonomy {between foliations} 
defined as follows:

\begin{definition}\label{d.invariant}
Let $\cF_1,\cF_2$ be foliations which are invariant under some diffeomorphism $f\in\Diff^1(M)$. {Let $M_1$ be an $\cF_1$-saturated subset of $M$.}
Assume that $\{m_x\}_{x\in M_1}$ is a Margulis system of measures on $\cF_1|_{M_1}$ and that  $\cF_2$ is  transverse to $\cF_1$.
The system $\{m_x\}_{x\in M_1}$ is \new{invariant}, respectively \new{quasi-invariant}, along $\cF_2$ if, for all $\cF_2$-holonomies $h:U\to V$ with $U,V$ contained in $\cF_1$-leaves {included in $M_1$:}
 \begin{equation}\label{eq-invariance}
   h_*(m_x|U)=m_{h(x)}|V \text{ for any }x\in U,
  \end{equation}
respectively:
 \begin{equation}\label{eq-quasi}
   h_*(m_x|U) \text{ and }m_{h(x)}|V\text{ are equivalent for any }x\in U.
  \end{equation}
\end{definition}

\begin{remark}
The quasi-invariance in \eqref{eq-quasi} can be characterized by the absolute continuity of the holonomies along $\cF_2$ with respect to a class of transversal measures defined by the Margulis system on $\cF_1$.
\end{remark}

{\begin{remark}
A system of measures along a foliation $\cF_1$ which is invariant along $\cF_2$ can be seen as a \emph{$\cF_2$-transverse measure} in the sense of Plante \cite{Plante75}, Ruelle \cite{Ruelle75}, and Ruelle-Sullivan \cite{RuelleSullivan75}. 

\end{remark}
}

{Though an arbitrary measurable system of measures along the strong unstable foliation does not need to correspond to the disintegration of any invariant probability measure, those we construct in this paper will  (see Proposition \ref{prop-LPS}).}

\section{Main Results}

This section collects our main results.
Our techniques deal with the following type of diffeomorphisms.
For convenience, we fix some Riemannian structure on the compact manifold $M$.

\begin{definition}  \label{flowtype}
A diffeomorphism $f:M\to M$ {is} {\bf  flow-type} if:
 \begin{enumerate}
  \item[\mftPH] {\sc partial hyperbolicity:} $f$ is strongly partially hyperbolic with splitting $TM=E^s\oplus E^c\oplus E^u$ and $\dim E^c=1$;
  \item[\mftDC] {\sc Dynamical coherence:} there are invariant foliations $\Fcs$ and $\Fcu$ tangent to $E^s\oplus E^c$ and $E^c\oplus E^u$;
 \end{enumerate}  
Let $\Fc$ be the foliation whose leaves are the connected components of the intersections $\Fcs(x)\cap\Fcu(x)$, $x\in M$. 
 \begin{enumerate}
  \item[\mftCLeaf] {\sc Center leaves:}  {The center foliation $\Fc$ is oriented and has at least one  compact  leaf.}
   \end{enumerate}  
  Let $F:\RR\times M\to M$ be the continuous flow along $\Fc$ with unit positive speed.
 \begin{enumerate}
  \item[\mftPush] {\sc Flow like dynamics:} there is a continuous $\tau:M\to\RR$ such that, for all $x\in M$, $f(x)=F(x,\tau(x))$ and $\tau(x)>0$. 
 \end{enumerate}
\end{definition}

{
\begin{remarks}
Let us note the following:

-- A diffeomorphism is flow-type if and only if its inverse is.

-- A flow obviously satisfies the flow like condition (IV). The converse is far from true (see the comments in the introduction).

-- The assumption of a compact leaf in condition (III) might seem artificial. It is only used to prove that the dilation of the unstable Margulis system is $D_u>1$ (Lemma~\ref{lem-Dbigger}) and there are many other sufficient conditions verified near the time-one map of an Anosov flow such as the uniform volume expansion/contraction of the center-unstable/center-stable leaves. However discretized Anosov flows have compact leaves (see \cite[Appendix G]{BFFP1}).

\end{remarks}
}

Following Margulis, we build special measures on strong stable and strong unstable leaves. More precisely, let $M^u:=M\setminus\Cyl^u$ where $\Cyl^u$  is the union of the unstable leaves that intersect some compact center leaf.
Define $M^s$ and $\Cyl^s$ likewise. Recall the notion of equivalence along some foliation from Sec.~\ref{secPH}.

\begin{theorem}
\label{mainthm-conditionals}
Let $f$ be a  flow-type $C^2$ diffeomorphism  {with}   minimal stable  foliation on a compact manifold $M$. Then  there is a {measurable} {system} of measures $\{m^u_x\}_{x\in M^u}$ such that:
 \begin{enumerate}
  \item\label{item-Radon} each $m^u_x$ is an atomless, fully-supported Radon measure on $\Fu(x)$  with respect to the intrinsic topology;
  \item\label{item-Margulis} $f_*m^u_x=D_u^{-1}m^u_{f(x)}$ for some constant $D_u>1$ (independant of $x$);
   \item\label{item-cs-qi} the system $\{m^u\}_{x\in M^u}$ is $cs$-quasi-invariant;
    \item\label{item-Mu}  $M^u$ is dense with full measure for any ergodic measure not supported on a closed leaf.
 \end{enumerate}
{In fact a stronger result than Item (3) holds:
For any $\delta > 0$ there exists $C(\delta) > 1$ such that if $A_1$ and $A_2$ inside $u$-leaves are $\delta$-equivalent along $\Fcs$, then 
\begin{equation}\label{eqLocalCSinv}
 \frac{1}{C(\delta)} < \frac{m^u(A_1)}{m^u(A_2)} < C(\delta).
\end{equation}
 }
\end{theorem}

Item \eqref{item-Margulis} says that $(m^u_x)_{x\in M^u}$ is a Margulis system on $\Fu|_{M^u}$ with dilation $D_u$ (see Def.~\ref{defMargulisSys}). We call  $\{m^u_x\}_{x\in M^u}$ the \emph{unstable Margulis system}. Considering $f^{-1}$, we  define similarly the measures $m^u_x$, called the  \emph{unstable Margulis conditionals}.

\begin{addendum}\label{addendum}
In the setting of the previous theorem, {assuming minimality of both stable and unstable foliations, the dilation factor $D_u$ of the unstable Margulis system satisfies:
$D_u=\exp h_\top(f)$.}
\end{addendum}

\medbreak

{The theorem will be proved in Section~\ref{s-Margulis} and the addendum in Section~\ref{s-dicho}. }

\begin{remarks}$ $

(1) The above theorem and addendum, applied to $f^{-1}$, define a \emph{stable Margulis system} $\{m^s_x\}_{x\in M^s}$ with dilation factor {$D_s=D_u^{-1}=\exp( - h_\top(f))<1$}.

(2) The $C^2$ smoothness assumption is only used by Theorem \ref{thm-Lcu-AC} to establish absolute continuity of the $s$-holonomy but it is probably enough to assume $C^{1+\alpha}$ smoothness. We do not know {about the $C^1$ smooth case}.

{
(3) The strong Margulis systems are unique up to a set negligible for the MMEs: e.g., for any two unstable Margulis systems $\{m^u_x\}$ and $\{\tilde m^u_x\}$ satisfying items (1) and (2) of Theorem~\ref{mainthm-conditionals}, the uniqueness of the disintegration of a probability measure yields $\mu(\{x:m^u_x\ne\tilde m^u_x\})=0$  for any MME $\mu$ with $\lambda^c(\mu)\leq0$. 
}
\end{remarks}

Using tools from Ledrappier and Young \cite{LYII}, we prove the following result in Section~\ref{s-dicho}.

\begin{theorem}\label{mainthm-mme}
Let  $f$ be a $C^2$ diffeomorphism that is  flow-type {and} minimal stable and unstable foliations  on a compact manifold $M$. {Let $\mu$ be an ergodic  MME.}

If  $\lambda^c(\mu)\leq 0$, then the disintegration of $\mu$ along $\Fu$ is given by  the unstable Margulis system $\{m^u_x\}_{x\in M^u}$ from Theorem~\ref{mainthm-conditionals}. 
In particular, the measure $\mu$ has full support.
\end{theorem}

{The above applied to $f^{-1}$ shows that an ergodic MME with $\lambda^c(\mu)\geq0$ has disintegration along $\Fs$ given by $\{m^s_x\}_{x\in M^s}$. In particular, any MME has full support.}

\begin{remark}
The above theorem gives \emph{more information} in the non-hyperbolic case. Indeed, if $\mu$ is an ergodic measure of maximal entropy with $\lambda^c(\mu)=0$, then the disintegrations, along \emph{both} strong foliations $\Fu$ and $\Fs$, are given by the corresponding  Margulis systems from Theorem~\ref{mainthm-conditionals}.
\end{remark}

The dichotomy will  follow from two results about hyperbolic measures of maximal entropy. The first is a uniqueness result, based on the Hopf argument.

\begin{proposition}\label{mainthm-hopf}
Let  $f$ be a flow-type $C^2$ diffeomorphism { and minimal stable and unstable strong foliations}  on a compact manifold $M$.
Let $\mu$ be some ergodic MME. 
If $\mu$ is hyperbolic, say $\lambda^c(\mu)<0$, then there is no other ergodic MME $\nu$  with $\lambda^c(\nu)\leq 0$.
\end{proposition}

The second result is a symmetry argument, using the one-dimensional center leaves.  It builds  so-called twin measures (see  \cite{RHRHTU, DGMR}).

{
\begin{proposition}\label{mainthm-twin}  
Let  $f$ be a $C^1$ diffeomorphism of a compact manifold $M$. Let $\Fc$ be an orientable one-dimensional foliation (with continuously varying $C^1$ leaves). Assume that, for all $x\in M$, $f$ maps $\Fc(x)$ to itself in an orientation-preserving way. Let $\mu\in\Proberg(f)$ satisfy:
 \begin{enumerate}
  \item its Lyapunov exponent along $\Fc$ is $\lambda_c(\mu)<0$;
  \item for $\mu$-a.e.\ $x\in M$, the following set is relatively compact in the intrinsic topology of $\Fc(x)$:
   $$
       \mathcal W^c(x):=\{y\in\Fc(x): \limsup_{n\to\infty} \frac1n\log d(f^nx,f^ny) < 0 \};
    $$
   \item and for $\mu$-a.e. $x\in M$, the leaf $\Fc(x)$ is noncompact and contains no fixed point.
 \end{enumerate}
Then there is another invariant probability measure  $\nu$ which is isomorphic to $\mu$ and with exponent $\lambda_c(\nu)\geq0$.
\end{proposition}
}

Finally, we  state the abstract version of our main result:
\begin{theorem}
\label{mainthm-dichotomy}
For any  flow-type $C^2$ diffeomorphism $f$ {with} minimal stable and unstable foliations on a compact manifold $M$,
we have the following dichotomy:
\begin{enumerate}
\item either all the measures of maximal entropy have zero central Lyapunov exponents, or
\item there are exactly two ergodic measures of maximal entropy where one has a positive central exponent and one has a negative central exponent, {and both are Bernoulli.}
\end{enumerate}
\end{theorem}

The next theorem shows that there is an abundance of diffeomorphisms satisfying the above assumptions. It  follows from properties of perturbations of {time-one maps of} {transitive} Anosov flows {established in \cite{BD96} and \cite{BDU}, as} discussed in Section~\ref{s-pertub-flow}.

\begin{theorem}\label{mainthm-mft}
If $\varphi^t$ is a  transitive Anosov flow on a compact 
manifold $M$, then {for all $T\neq 0$} there exists a {$C^1$} open set $\mathcal U$ in $\Diff^1(M)$ such that $\phi^T$ belongs to the $C^1$-closure of $\mathcal U$ and every $f\in\mathcal U\cap \mathrm{Diff}^2(M)$ {is}  flow-type with both stable and unstable foliations minimal.
\end{theorem}

\section{Building Margulis systems of measures}\label{s-Margulis}

In this section,  we consider flow-type diffeomorphisms
whose {strong stable foliation is minimal}. To begin with, we follow Margulis' construction of a system of measures on the $cu$-leaves that are invariant  under stable holonomies.
We then deduce from this a {system} of $u$-conditionals that are quasi-invariant under center-stable holonomies. This will prove Theorem~\ref{mainthm-conditionals}. 

\subsection{The $cu$-conditionals}\label{s-cu-conditionals}
We  follow Margulis' construction.

{\begin{proposition} \label{cumargulis}
Let $f \in \Diff^2(M)$ on  a compact manifold $M$ with a  dominated splitting $E^s \oplus E^{cu}$ with $E^s$ uniformly contracted.  Assume that:
 \begin{enumerate}
 \item there are foliations $\cF^{cu}$ and $\cF^s$ which are tangent to, respectively  $E^{cu}$ and $E^s$;
 \item $\cF^s$ is minimal.
 \end{enumerate} 
 Then there is a Margulis system $\{m_x^{cu}\}_{x \in M}$ on $\mathcal{F}^{cu}$ which is invariant under $\Fs$-holonomy and such that each $m^{cu}_x$ is atomless, Radon, and fully supported on $\cF^{cu}(x)$. 
 
 {If some $cu$-leaf contains an invariant topological circle, then the dilation constant of this Margulis system is $D_u>1$.}
\end{proposition}}

We introduce some convenient definitions. Let  $\sigma\in \{c, cu, cs, s, u\}$. We denote by  $\lambda^\sigma$  the intrinsic Riemannian volume on each $\sigma$-leaf: for any subset $E$ of a $\sigma$-leaf, $\lambda^\sigma(E)$ is its volume with respect to the Riemannian structure on the leaf. We denote by $d_\sigma$ the distance defined on each leaf by the induced Riemannian structure and defining the intrinsic topology.

The $\sigma$-balls are $B^\sigma(x,r):=\{y\in\cF^\sigma(x):d_\sigma(y,x)<r\}$. For a subset $A$ of such a leaf, we set $B^\sigma(A,r):=\bigcup_{x\in A} B^\sigma(x,r)$.
A \textit{$\sigma$-subset} is a $d_\sigma$-bounded subset of a $\sigma$-leaf. A \textit{$\sigma$-test function} is a nonnegative function $\psi:M\to\RR$ such that $\{x\in M:\psi(x)\ne0\}$ is a $\sigma$-subset and the restriction of $\psi$ to 
$$\supp(\psi):=\overline{\{x\in M:\psi(x)\ne0\}}$$ 
is continuous. 
We write $\psi>0$ if {$\psi\geq0$ and} $\{\psi>0\}$ has non-empty interior in the intrinsic topology. We denote by $\cT^\sigma$ the collection of all $\sigma$-test functions.

{We recalled the notions of holonomy, holonomy size, and holonomy equivalence for sets in Sec.~\ref{secPH}.} 
Two functions $\psi,\phi$ are $\eps$-equivalent along $\cF^\sigma$ if their supports are equivalent through a $\sigma$-holonomy $h$ with size at most $\eps$ and satisfying  $\phi=\psi\circ h$.
Two  submanifolds $N',N''$ are \emph{$t$-transverse} if they are transverse and if for every $x\in N'\cap N''$, the angle between any two nonzero vectors in $T_xN'$ and $T_xN''$ is at least $t$.

 \medbreak
 
Following Margulis, we  consider  functionals $\lambda:\cT^{cu}\to\RR$. {Note that $\lambda^{cu}$ {defines} one such functional. The map $f$ acts on them by:
 $$
   \forall\psi\in\cT^{cu}\quad   f(\lambda)(\psi) := \lambda(\psi\circ f^{-1}).
 $$
}

A key class of such functionals are $\ell_n:=f^n(\lambda^{cu})$ for any $n\in\NN$. That is, for any $\phi\in\cT^{cu}$, 
 $$
    \ell_n(\phi):=\int \phi\circ f^{-n}\, d\Lcu.
 $$
To normalize, we fix some $\phi_1\in\mathcal T^{cu}$ with $\phi_1>0$. Considering the topology of pointwise convergence (i.e., working in  $\RR^{\cT^{cu}}$ with the product topology),  let $\LL$  be the closure of the following set:
 $$
   \LL_1:=\{\lambda=\sum_{i=1}^n c_i\ell_{t_i}:n\in\NN^*,\; t_1,\dots,t_n\in\NN,\; c_1,\dots,c_n>0
    \text{ with }\lambda(\phi_1)=1\}.
 $$

{We will use the following covering numbers. For $K$,  a $cu$-subset, and $\rho>0$,  we denote by $r^{cu}(K,\rho)$ the smallest  integer $k\geq0$ such that there are $x_1,\dots,x_k\in K$ with $K\subset\bigcup_{i=1}^k B^{cu}(x_i,\rho)$.}

\medbreak

We build the $cu$-system as a functional.

\begin{proposition}\label{proposition-functional}
{Let $f\in\Diff^2(M)$ be partially hyperbolic with a splitting $TM=E^s\oplus E^{cu}$ and an invariant foliation $\Fcu$ tangent to $E^{cu}$. Assume that the stable foliation is minimal. Then}
there exist $\Lambda\in\LL$ and $D_u> 0$ such that 
 $$
     f(\Lambda)=D_u\cdot\Lambda
  $$
and, for some positive numbers $C,R$, for any $\phi\in\cT^{cu}$:
 \begin{enumerate}[(a)]
  \item $\Lambda(\phi)\leq Cr^{cu}(\supp\phi,R)\|\phi\|_\infty$;
  \item if  $\phi>0$, then $\Lambda(\phi)>0$; and
  \item if $\psi\in\cT^{cu}$ is $s$-equivalent  to $\phi$, then $\Lambda(\phi)=\Lambda(\psi)$.
 \end{enumerate}
\end{proposition}

To prove this, we will show that $\LL$ is a convex and compact set and then apply the Schauder-Tychonoff fixed point Theorem to a normalized action of $f$ on $\LL$.  We note that {$\LL$} is not a Banach space, but this is not required to apply the Schauder-Tychonoff fixed point Theorem.

We will relate the iterations of different $cu$-test functions by using the invariance under $s$-holonomy and especially the following theorem (see, e.g., \cite[Theorem C]{AbdenurViana}).

\begin{theorem}\label{thm-Lcu-AC}
Let  $f$ be a $C^2$ diffeomorphism  on a compact Riemannian manifold  $M$.  Assume that there is a dominated splitting $E^s\oplus E^{cu}$ with $E^s$ uniformly contracting.  Fix $t>0$ and let $A_1,A_2$ be two submanifolds $t$-transverse to $E^s$ and $s$-equivalent through $h:A_1\to A_2$. Then $h$ is absolutely continuous. 

More precisely, writing $\lambda_1,\lambda_2$ for the Riemannian volume on $A_1,A_2$, the measures $h_*\lambda_1$ and $\lambda_2$ are equivalent and there are constants $C<\infty$ and $\alpha>0$ depending only on $f$,  $E^s$, and $t$ such that, letting $\eps$ be the size of the holonomy $h$:
 \begin{equation}\label{eq-Lcu-AC}
      \left|\frac{d h_*{\lambda_1}}{d\lambda_2} -1 \right|\leq {C^{1+\eps} \eps^\alpha}.
 \end{equation}
\end{theorem}

Note that \cite[Theorem C]{AbdenurViana} gives the usual (stronger) bound by $C_0\eps^\alpha$ but only for $0\leq \eps\leq\eps_0$. The above bound for all $\eps\geq0$ is obtained by a routine computation by decomposing a holonomy of some large size $\eps\geq\eps_0$ into a product of $\lceil \eps/\eps_0\rceil$ holonomies of size at most $\eps_0$.

\medbreak

{In the rest of Subsection~\ref{s-cu-conditionals}, we assume that  $f\in\Diff^2(M)$ is partially hyperbolic with a splitting $TM=E^s\oplus E^{cu}$ and an invariant foliation $\Fcu$ tangent to $E^{cu}$ and a minimal strong stable foliation $\Fs$.}

We need two additional lemmas. The first one will give a uniform bound on the volume growth in the center unstable leaves.  

\begin{lemma}\label{lem-Lcu-local}
For any open, non-empty $cu$-subset $A$, there are constants $C_A<\infty$ and $r_A>0$ such that, 
 \begin{equation}\label{eq-Lcu-local}
    \forall x\in M\;\forall n\geq0\; \Lcu(f^n B^{cu}(x,r_A)) \leq C_A \Lcu(f^n A).
 \end{equation}
\end{lemma}

\begin{proof}
Fix $r_1>0$ so small that $A_1:=A\setminus B^{cu}(\partial A,r_1)$ is not empty. The minimality implies that any $x$ is $s$-equivalent to some point in $A_1$. 
{For every $x\in M$,} the continuity of the foliation $\Fs$  and its transversality to $A$ yield  $r_x>0$ and $R_x<\infty$ such that, for any $y\in B(x,r_x)$, $B^{cu}(y,r_x)$ is $(R_x,s)$-equivalent to a subset of {$A$}. 
The compactness of $M$ {allows to choose $r_A>0$ and $R_A<\infty$ independent of $x\in M$.}

Since $\Fs$ is contracted, there are numbers $C<\infty$ and $\kappa<1$ such that, {for all $n\geq0$ and all $x\in M$,} the set $f^n(B^{cu}(x,r_A))$ is $(C\kappa^nR_A,s)$-equivalent to a subset of $f^n(A)$. {We conclude by} Theorem \ref{thm-Lcu-AC}.
\end{proof}

\begin{corollary}\label{cor-unif-sup}
For any $\phi\in\cT^{cu}$ with $\phi>0$, there are numbers $C(\phi)<\infty$ and $R(\phi)>0$ such that for any $\psi\in\cT^{cu}$  and any $n\geq 0$:
 $$ 
     \int \psi\circ f^{-n}\, d\Lcu \leq C(\phi)r^{cu}(\supp\psi,R(\phi))\|\psi \|_\infty \int \phi\circ f^{-n}\, d\Lcu.
 $$
\end{corollary}

\begin{proof}
The left hand side is bounded by $\|\psi \|_\infty\cdot\lambda^{cu}(f^n(\supp\psi))$.  Fix some $0<t<\sup \phi$. 
The previous lemma with $A:=\{\phi>t\}$ yields $r_A>0$ and $C_A<\infty$.
Since $\supp(\phi)$ is compact in its $cu$-leaf $\Fc(A)$, there are $x_1,\dots,x_N\in\Fc(A)$ with $N=r^{cu}(\supp\psi,r_A)$ such that  $\supp(\psi)\subset \bigcup_{i=1}^N B^{cu}(x_i,r_A)$. 
Now 
 $$
   \lambda^{cu}(f^n(B^{cu}(x_i,r_A)))\leq  C_A \lambda^{cu}(f^nA) \leq C_A \lambda^{cu}(\phi\circ f^{-n})/t.
  $$
Summing over the cover of $\supp\psi$, the claim follows with $C(\phi):=C_A/t$ and $R(\phi):=r_A$.
\end{proof}

The next lemma establishes approximate holonomy invariance.

\begin{lemma}\label{lem-near-hol-inv}
There are numbers $C<\infty$ and {$0<\kappa<1$} with the following properties.
Let $\psi,\phi\in\cT^{cu}$ be $(s,\Delta)$-equivalent for some $\Delta<\infty$.

First,  if $\psi>0$, then, for all $n\geq0$:
 \begin{equation}\label{eq-near-hol1}
   \forall \lambda\in\LL\quad \left|\lambda(\psi\circ f^{-n}) - \lambda(\phi\circ f^{-n})\right|\leq {C^{1+\Delta\kappa^n}\kappa^{\alpha n} } \Delta^\alpha \lambda(\psi\circ f^{-n}).
 \end{equation}
Second,  for any $\phi_1\in\cT^{cu}$ with $\phi_1>0$, there are numbers $C(\phi_1)$ and $R(\phi_1)>0$ such that,  for any $n\geq 0$ we have
  \begin{equation}\label{eq-near-hol2}
    \left|\ell_n(\psi) - \ell_n(\phi)\right|
        \leq {C(\phi_1)^{1+\Delta\kappa^n}} r^{cu}(\supp(\psi),R(\phi_1))\|\psi \|_\infty {\kappa^{\alpha n}}  \Delta^\alpha \ell_n(\phi_1).
   \end{equation}
\end{lemma}

\begin{proof}
Since $\phi$ and $\psi$ are $\Delta$-equivalent, there is $h:\supp(\psi)\to\supp(\phi)$ with size at most $\Delta$ such that $\psi=\phi\circ h$. Since $\Fs$ is uniformly contracted, $\phi\circ f^{-n}$ and $\psi\circ f^{-n}$ are $C\kappa^n\Delta$-equivalent through $h_n:=f^n\circ h\circ f^{-n}$  for some $C<\infty$ and $\kappa<1$. Theorem \ref{thm-Lcu-AC}  yields (replacing $C$ by a larger constant each time needs arise):
 $$\begin{aligned}
    |\ell_n(\phi)&-\ell_n(\psi)|=\biggl|\int \phi\circ f^{-n} d\Lcu_{f^ny} - \int \psi\circ f^{-n} d\Lcu_{f^nx}\biggr| \\
    &= \left|\int \left(\phi\circ f^{-n}\circ h_n \frac{d (h_n)_*\Lcu_{f^ny}}{d \Lcu_{f^nx}} -\psi\circ f^{-n}\right)\, d\Lcu_{f^nx}\right|\\
    &{= \left|\int \psi\circ f^{-n}  \left( \frac{d (h_n)_*\Lcu_{f^ny}}{d \Lcu_{f^nx}} - 1\right)\, d\Lcu_{f^nx}\right|}\\
    &\leq { C^{1+C\Delta\kappa^n} \Delta^\alpha \kappa^{\alpha n} } \int |\psi\circ f^{-n}| d\Lcu_{f^nx}\\
    &\leq C^{1+\Delta\kappa^n} \Delta^\alpha \kappa^{\alpha n} \ell_n(|\psi|).
   \end{aligned} $$
Let $\lambda\in\LL_1$ so that $\lambda=\sum_{i=1}^I c_i\ell_{t_i}$. Using the identity $\ell_t(\phi\circ f^{-n})=\ell_n(\phi\circ f^{-t})$, we get:
  $
    \lambda(\phi\circ f^{-n}) 
       = \sum_{i=1}^I c_i \ell_{n}(\phi\circ f^{-t_i}).
   $
Hence, if $\psi>0$, 
  $$
     |\lambda(\phi\circ f^{-n})-\lambda(\psi\circ f^{-n})| \leq { C^{1+\Delta\kappa^n} \Delta^\alpha \kappa^{\alpha n} } \lambda(\psi\circ f^{-n}),
   $$
proving eq.~\eqref{eq-near-hol1} for all $\lambda\in\LL$ by continuity. 

Let $\phi_1\in\cT^{cu}$ with $\phi_1>0$. Applying Corollary~\ref{cor-unif-sup}, we obtain $C(\phi_1)$ and $R(\phi_1)>0$ such that eq.~\eqref{eq-near-hol2} holds.
\end{proof}

\begin{proof}[Proof of Proposition \ref{proposition-functional}]
{We check the claims (a), (b), and (c) in turn.}

\newcommand\step[2]{\medbreak\noindent{\bf Step #1:} \emph{#2}\smallbreak}

\step{1}{Claim (a): $\exists R>0$ $\forall \lambda\in\LL$ $\lambda(\psi)\leq C r^{cu}(\supp\psi,R)\|\psi\|_\infty$.}

Corollary~\ref{cor-unif-sup} {applied to the test function $\phi=\phi_1$ (used in the normalization)} yields numbers $C<\infty$ and $R>0$ such that  for any $\psi\in\cT^{cu}$:
 $$
    \forall n\geq0\; \ell_n(\psi) \leq C r^{cu}(\supp\psi,R)\|\psi\|_\infty \ell_n(\phi_1).
  $$
Therefore, for  any $\lambda\in\LL_1$ :
\begin{equation}\label{eq-lambda-upper}\begin{aligned}
    \lambda(\psi)  =\sum_{i=1}^I c_i\ell_{t_i}(\psi) &\leq \sum_{i=1}^I c_i C r^{cu}(\supp\psi,\rho)\|\psi\|_\infty  \ell_{t_i}(\phi_1)\\
    &=C r^{cu}(\supp\psi,R)\|\psi\|_\infty \lambda(\phi_1).
 \end{aligned}\end{equation}
{Since $\lambda(\phi_1)=1$, this proves the claim,  first for $\lambda\in\LL_1$, then for $\LL=\overline{\LL_1}$ by continuity. Step 1 is done.}

\medbreak

Note that eq.~\eqref{eq-lambda-upper} implies that $\LL_1$ is  a subset of the compact set
 $$
 \prod_{\psi\in\cT^{cu}} [0,C\cdot r^{cu}(\supp\psi,R)\|\psi\|].
 $$ 
Hence its closure $\LL$  is compact.

\step{2}{Claim~(b): $\forall \psi\in\cT^{cu}$ if $\psi>0$ there is $C(\psi)>0$ s.t. $\forall\lambda\in\LL$, $\lambda(\psi)\geq C(\psi)$.}

We now assume that ${\psi}>0$ so we can apply Corollary~\ref{cor-unif-sup} exchanging $\psi$ and $\phi_1$. We get new numbers $C'$ and $R'$, {depending on $\psi$}.
Setting $C_1:=C'r^{cu}(\supp\phi_1,R')$, we have, for all $n\geq0$, $\ell_n(\phi_1)\leq C_1\ell_n(\psi)$. That is, $\ell_n(\psi)\geq(1/C_1)\ell_n(\phi_1)$ so that, for any $\lambda\in\LL$:
  \begin{equation}\label{eq-lambda-lower}
   \lambda(\psi)\geq \frac1{C_1} \lambda(\phi_1) = \frac1{C_1},
 \end{equation}
concluding Step 2 with the lower bound $1/C_1(\psi)$.

\step{3}{Existence of $\Lambda\in\LL$ with $f(\Lambda)=D\cdot\Lambda$}

We now  build the functional $\Lambda$ as a fixed point of the map
 $$
   \bar f:\LL\to\LL,\quad \lambda\longmapsto\frac{\lambda(\cdot\circ f^{-1})}{\lambda(\phi_1\circ f^{-1})}.
  $$

We claim that $\bar f$ is well-defined and continuous from $\LL$ to $\RR^{\cT^{cu}}$.
Indeed, the map  $\lambda\mapsto \lambda(\phi_1\circ f^{-1})$ from $\LL$ to $\RR$ is well-defined since $\phi_1\circ f^{-1}\in\cT^{cu}$, is obviously continuous, and  is positive by Step~2. 
Moreover, $\lambda(\,\cdot\circ f^{-1}):\cT^{cu}\to\RR$ is well-defined and $\lambda\mapsto\lambda(\,\cdot\circ f^{-1})$ is continuous from $\LL$ to $\RR^{\cT^{cu}}$. The claim is proved. 

Finally, it is obvious that $\bar f(\LL_1)=\LL_1$, hence $\bar f:\LL\to\LL$ is a well-defined continuous map. 
Since $\LL$ is a convex, compact subset  of the locally convex linear space {$\RR^{\cT^{cu}}$}, the Schauder-Tychonoff Theorem applies and yields $\Lambda\in\LL$ with $\bar f(\Lambda)=\Lambda$. 

In other words, $f(\Lambda) = D \cdot \Lambda$ with $D= \Lambda(\phi_1\circ f^{-1})>0$.

\step{4}{Claim (c): $s$-holonomy invariance of $\Lambda$}

Let $\phi,\psi\in\cT^{cu}$. Observe that positive (respectively negative) parts of $\phi$ and $\psi$ must then be $s$-equivalent. Thus we can assume $\phi,\psi>0$. Assume that $\psi$ and $\phi$ are $s$-equivalent. By compactness of their support, they are $(s,\Delta)$-equivalent for some $\Delta>0$ and therefore $\phi\circ f^{-n}$ and $\psi\circ f^{-n}$ are $(s,C\kappa^n\Delta)$-equivalent with $0<\kappa<1$. Using the dilation by $D\ne 0$ and eq.~\eqref{eq-near-hol1} from Lemma~\ref{lem-near-hol-inv} (approximate holonomy invariance), we get that, for any $\eps>0$, for large enough $n\geq0$:
 $$\begin{aligned}
    |\Lambda(\phi)-\Lambda(\psi)| &= D^{-n}|\Lambda(\phi\circ f^{-n})-\Lambda(\psi\circ f^{-n})|\\
     &\leq D^{-n}\eps\cdot\Lambda(\psi\circ f^{-n}) = \eps\cdot\Lambda(\psi)
    \end{aligned}$$
 As $\eps>0$ is arbitrarily small this implies $\Lambda(\phi)=\Lambda(\psi)$ and concludes the step.
 
 \medbreak
 Proposition  \ref{proposition-functional} is established.
\end{proof}

We deduce a Margulis system of $cu$-measures from the functional $\Lambda$.

\begin{proof}[Proof of Proposition \ref{cumargulis}]
Proposition \ref{proposition-functional} yields  a functional $\Lambda$ on $\cT^{cu}$, which contains $C_c(\Fcu(x))$ for all $x\in M$. {We note that, for each $x\in M$, $\Fcu(x)$ is a locally compact metric space and} $\Lambda|C_c(\Fcu(x))$ is linear and positive (because this holds for all $\lambda\in\LL_1$ and extends by continuity to $\LL$). 
Hence, Riesz's Representation Theorem \cite[2.14]{Rudin} gives a {Radon} measure $m_x$ on $\Fcu(x)$, for each $x\in M$, characterized by:
 $$
     \forall \phi\in C_c(\Fcu(x)),\; m_x(\phi) = \Lambda(\phi).
 $$
The full support, and $s$-invariance of each $m_x$ follows from properties  (b) and (c) from Proposition \ref{proposition-functional}.

We prove that each $\nucu_x$ is atomless using the holonomy invariance. Indeed, assume by contradiction that there is $y\in\Fcu(x)$ with $\nucu_x(\{y\})>0$. Consider the stable leaf $\Fs(y)$. By assumption it is dense in $M$, hence by transversality in $\Fcu(x)$. By $s$-invariance, $\nucu_x$ must have a dense set of atoms $z\in\Fcu(x)$, all of which have measure $\nucu_x(\{y\})>0$. This contradicts the finiteness of $\nucu_x$ on compact sets.

We finally deduce the continuity from the holonomy invariance. As $\mathcal{F}^{cu}$ and $\mathcal{F}^s$ are transverse, for any $x_0\in M$, there is a neighborhood $B$ of $x_0$ and a continuous map $h_0:B\times\Fcu_B(x_0)\to B$ with $\{h_0(x,y)\}=\Fcu_\loc(x)\cap\Fs_\loc(y)$ (in particular, $h_0(x_0,y)=y$). Let $\phi\in C(M)$. By holonomy invariance, $\nucu_x(\phi|\Fcu_B(x))=\nucu_{x_0}(\phi\circ h_0(x,\cdot))$. Hence,
 $$
    |\nucu_{x}(\phi|\Fcu_B(x)) -\nucu_{x_0}(\phi|\Fcu_B(x_0))| \leq \int_{\Fcu_B(x_0)} |\phi\circ h_0(x,\cdot)-\phi |\, d\nucu_{x_0}
 $$
which converges to $0$ as $x$ goes to $x_0$. This is the continuity property.

{The assertion about  the dilation factor is $D_u>1$ follows from the next  Lemma.}
\end{proof}

\begin{lemma}\label{lem-Dbigger}
Let $f\in\Diff^2(M)$ have a dominated splitting $TM=E^s\oplus E^{cu}$ with tangent foliations $\Fs$ and $\Fcu$. Assume that $\Fs$ is minimal and that there is a Margulis system  $\{m^{cu}_x\}_{x\in M}$ carried by $\Fcu$.

If there is a invariant circle $\gamma$ contained in some $cu$-leaf, then 
the dilation of the Margulis system on $\Fcu$  satisfies: $D_u>1$.
\end{lemma}

\begin{proof}
Let   $\gamma$ be the invariant circle and let  $\Fcu(x)$ be its $cu$-leaf endowed with its intrinsic topology.
Since $\gamma$ is a topological attractor for the restriction of $f^{-1}$ to the invariant set $\Fcu(x)$, there is a relatively compact neighborhood $U$ of $\gamma)$ in $\Fcu(x)$ such that $U\setminus f^{-1}(U)$ has non empty interior. 
As $\nucu_{x}$ has full support in $\Fcu(x)$, it follows that:
$$
  D_u^{-1}\nucu_x(U)=\nucu_x(f^{-1}U)<\nucu_x(U),
 $$
proving $D_u>1$.
\end{proof}


\subsection{Building the $u$-conditionals}\label{s-u-conditionals}


We complete the proof of Theorem~\ref{mainthm-conditionals}.

We start with the previously built Margulis $cu$-system $\{m^{cu}_x\}_{x\in M}$ and define the family of measures $\{\nuu_x\}_{x\in M^u}$ by extending subsets of $u$-leaves to subsets of $cu$-leaves along the center foliation. For flows, Margulis used the formula 
 $$
   \nuu_x(A)=\nucu_x\left(\bigcup_{0\leq t<t_0}\phi^t(A)\right).
 $$ 
Here $t_0>0$ is chosen small so that disjoint subsets of the same $u$-leaf correspond to disjoint subsets of the $cu$-leaf. This ensures that $\nuu_x$ {inherits the $\sigma$-}additivity of $\nucu_x$. 

In our setting there is no flow commuting with the dynamics and we have to proceed differently. To keep the equivariance, we will replace $(\phi^t(x))_{0\leq t<t_0}$ by the $c$-segment $I_c(x)$ ``between $x$ and $f(x)$''. Such center curves cannot be assumed to be arbitrarily short, so even restricting $x$ to a small subset of a single $u$-leaf, the curves $I_c(x)$ might intersect and destroy the additivity property. It turns out that this problem only occurs on $\Cyl^u$, i.e., the union of the $cu$-leaves containing a compact center leaf. This is why we will build $u$-measures $m^u_x$ only for $x\in M\setminus\Cyl^u$.\footnote{Alternatively, one could consider a covering of $M$ where all center leaves are noncompact. Note also that this problem could be altogether avoided by restricting to perturbation of time $t$ maps for small enough $t$, instead of taking $t=1$.}

\subsubsection*{Preparations}
We recall or establish some useful properties of the dynamical foliations and of the $cu$-system built in Section~\ref{s-Margulis}. In particular, it allows us to disregard the closed center leaves.

{
\begin{remark}\label{rem-countable-closed}
There can be at most countably many compact center leaves, since each one is normally hyperbolic (see \cite[Theorem 4.1(b)]{HPS}).
\end{remark}
}

\begin{lemma} \label{circles}
Let $f\in\Diff(M)$ be a flow-type diffeomorphism. If a center leaf meets some unstable leaf in more than one point, then it is contained in the $cu$-leaf of a compact center leaf.
\end{lemma}

\begin{proof}
To simplify notations a bit, we prove the symmetric statement involving stable and center stable manifolds.

Let $x,y$ be two distinct points with $y\in\Fs(x)\cap\Fc(x)$ and $x\leq y$ along the center leaf (recall that $f$ being flow-type, it maps each center leaf to itself preserving some orientation). Let $L$ be the center segment $[x,y]_c$.  We are going to show that $f^n(L)$ converges to a closed center leaf $\gamma$ as a compact subset with respect to the intrinsic distance in $\Fcu(x)$. The existence  of $\gamma$ in $\Fcu(x)$ will prove the lemma.

There are two cases, depending on the position of $y$ with respect to $f(x)$. In the first case, we have: 
 \begin{equation}\label{eq-case1}
 x\leq f(x) \leq y\leq f(y)
 \end{equation}
{hence} the decompositions (the unions are disjoint up to one point):
 $$\begin{aligned}
    &L{=[x,y]_c}=[x,f(x)]_c \cup [f(x),y]_c\\
    &f(L)=[f(x),f(y)]_c=[f(x),y]_c \cup[y,f(y)]_c.
  \end{aligned}$$
For large enough $k\geq0$, the points $f^k(x)$ and $f^{k}(y)$ are arbitrarily close while  the length of $[f^k(x),f(f^k(x))]_c$  is bounded independently of $k$. We can assume that this holds already for $k=0$. Thus there is a $s$-holonomy  $h:[x,f(x)]_c\to[y,f(y)]_c$ such that $h(x)=y$. 

Now define the map $\psi:L\to f(L)$ by $\psi|[x,f(x)]_c=h$ and $\psi|(f(x),y]_c=Id$.
Observe that it is a bijection such that $\psi(x)\in\Fs_\loc(x)$ and that there is a finite bound $\delta:=\sup_{z\in L}d(z,\psi(z))$ on the Hausdorff distance $d_H(L,f(L))$.

This implies that, for any $n\geq0$,
 $$
   d_H(f^n(L),f^{n+1}(L))\leq \sup_{z\in L} d(f^n(z),f^n(\psi(z)))
 $$
{decays exponentially fast since $\Fs$ is uniformly contracted}. It follows that $$\sum_{n\geq0} d_H(f^n(L),f^{n+1}(L))<\infty$$ and so $f^n(L)$ converges to some nonempty compact $\gamma$ in the complete space  of all nonempty compact subsets of $\Fcu(x)$. {Obviously $\gamma$ is a closed center leaf.}
This proves the lemma in the case \eqref{eq-case1} holds. In the other case,
 $$
    x\leq y< f(x)\leq f(y).
  $$
As above, perhaps after replacing $x,y$ by some forward iterates, we find a $s$-holonomy $h:[x,f(x)]_c\to [y,f(y)]_c$ such that $h(x)=y$. By transversality of $\Fc$ and $\Fs$, the length of $[z,h(z))_c$ is lower bounded {for $z\in M$}. Hence $h^{N}(x)<f(x)\leq h^{N+1}(x)$ for some $N\geq1$, and one has:
 $$\begin{aligned}
    &L = [x,h^{-N}(f(x))]_c \cup [h^{-N}(f(x)),y]_c\\
    &f(L) = [f(x),h^N(y)]_c \cup [h^N(y),f(y)]_c
 \end{aligned}$$
and the following maps are $s$-holonomies:
 $$\begin{aligned}
   & h^N:[h^{-N}(f(x)),y]_c\longrightarrow [f(x),h^N(y)]_c\\
   & h^{N+1}:[x,h^{-N}(f(x))]_c  \longrightarrow [h^N(y),f(y)]_c
  \end{aligned}$$
(since $h(f(x))=f(h(x))=f(y)$).

As before, we conclude that $f^n(L)$ converges to a compact center leaf in $\Fcu(x)$.
\end{proof}

{Center-unstable leaves have a simple structure:}

\begin{lemma}
Let $f\in\Diff^1(M)$ be a flow-type diffeomorphism.
For any $x\in M$,
 $$
    \Fcu(x) =\bigcup_{y\in\Fc(x)} \Fu(y).
  $$
{In particular, $\Cyl^u = \bigcup_{\gamma} \Fu(\gamma)$ where $\gamma$ ranges over the countably many compact center leaves.}
\end{lemma}

This is \cite[Prop. 2.9]{SY}. Note that the statement there assume that $f$ is a perturbation of some Anosov flow, but inspection of the proof shows that  flow-type is sufficient. Now we can show that the exceptional {sets $\Cyl^u$ and $\Cyl^s$ are negligible.} 

\begin{lemma}\label{lem-CuNull}
{Let $f$ be  a diffeomorphism on a compact manifold with flow-type. Then the following holds:
 \begin{enumerate}[\qquad(1)]
 \item\label{item-muCsnull} For any $\mu\in\Proberg(f)$ not carried by a compact center leaf, $\mu(\Cyl^{s})=\mu(\Cyl^u)=0$. 
 \item\label{item-nucuCsnull}  If $f$ admits a Margulis system of $cu$-measures $\{\nucu_x\}_{x\in M}$ with dilation factor $D_u\ne 1$ for any $x\in M$, {any compact center leaf has zero $\nucu_x$-measure and more generally $\nucu_x(\Cyl^s)=0$.}
 \end{enumerate}
 }
\end{lemma}

{
We will deduce from item~\eqref{item-nucuCsnull}  the unstable Margulis conditionals are atomless.
}

\begin{proof}
{
We prove item \eqref{item-muCsnull} by contradiction. Let $\mu\in\Proberg(f)$ be such that $\mu(\Cyl^s)>0$.  {The previous lemma {and $\sigma$-additivity} imply that  $\mu$ gives positive measure  to $\Fs(\gamma)$ for some compact center leaf $\gamma$. By ergodicity, we have $\mu(\Fs(\gamma))=1$.}  Iterating forward, this implies that $\mu(\gamma)=1$. This contradicts the assumption and proves $\mu(\Cyl^s)=0$. Iterating backwards yield $\mu(\Cyl^u)=0$.

We turn to item \eqref{item-nucuCsnull}. 
First let $\gamma$ be a compact center leaf. Since it is fixed, $m^{cu}_x(\gamma)=D_u \cdot \nucu_x(\gamma)$. As $D_u\ne 1$ and $\nucu_x(\gamma)<\infty$ (by the Radon property), we obtain $\nucu_x(\gamma)=0$, the first claim.

Now let us prove that $\nucu_x(\Fs(\gamma))=0$ for all arbitrary compact center leaves $\gamma$. Fix such a leaf $\gamma$. Since $\gamma\subset\Fs(\gamma)\cap\Fu(\gamma)$, we can find a countable cover
 $$
    \{y\in\Fcu(x): y\in\Fs(\gamma)\} \subset \bigcup_{i\geq1} h^s_i(\gamma\cap U_i)) 
 $$
where $h^s_i:U_i\to \Fcu(x)$, $i\geq1$, are  $s$-holonomies with $U_i\subset\Fcu(\gamma)$. Now, by $s$-invariance of the $cu$-system and the fact that $\nucu_x(\gamma)=0$, each term on the right hand side has zero $m^{cu}_x$-measure. This shows that $\nucu_x(\Fs(\gamma))=0$

The claim $\nucu_x(\Cyl^s)=0$ follows by $\sigma$-additivity and the previous lemma.
}
\end{proof}

\begin{proof}[Proof of Theorem \ref{mainthm-conditionals} ]
We begin by defining the measures $m^u_x$ along the unstable foliation.  For each $x\in  M^u$ and Borel subset $A\subset\Fu(x)$ we let
 $$
    m^u_x(A) = m^{cu}_x(\hat A) \text{ with } \hat A:=\bigcup_{y\in A} [x,f(x))_c.
 $$
This is well-defined since $\hat A$ is measurable when $A$ is. Consider a $u$-leaf $\Fu(x)$ for some  $x\in M^u$. Let $A_1,A_2,\dots$ be pairwise disjoint measurable sets. By Lemma~\ref{circles}, 
the extensions $\hat A_1,\hat A_2,\dots$ are also pairwise disjoint.
Hence, 
 $$
   m^u_x(\bigsqcup_n A_n)=m^{cu}_x(\widehat{\bigsqcup_n A_n})=m^{cu}_x(\bigsqcup_n \hat A_n) =\sum_n m^u_x(A_n),
 $$
proving the $\sigma$-additivity. Obviously, $m^u_x(\emptyset)=0$. Thus $m^u_x$ is a measure. 

\medbreak

{$\Cyl^u$ is a countable union of submanifolds with positive codimension. Thus the Baire theorem shows that $M^u:=M\setminus\Cyl^u$ is  dense. Lemma~\ref{lem-CuNull} shows that $\Cyl^u$ has zero measure for any ergodic measure {not concentrated on a single compact leaf.} Item~\eqref{item-Mu} is proven.}

 \medbreak
 
Observe that relative compactness, respectively nonempty interior, hold for $\hat A$ if it holds for $A$ yielding that $m^u_x$ is Radon and fully supported. {Item~\eqref{item-nucuCsnull} of} Lemma~\ref{lem-CuNull} implies that {any center leaf is $m^{cu}_x$-negligible for all $x\in M$, hence} each $m^u_x$ is atomless. {Item~\eqref{item-Radon} follows.}

\medbreak

We prove that  $\{m^u_x\}_{x\in M\setminus\Cyl^u}$ is a Margulis system. Observe that   $\{\nucu_x\}_{x\in M}$ is a Margulis {system with dilation $D_u>1$} by Proposition~\ref{cumargulis}.  {To determine the  dilation of $\{m^u_x\}$}, note that $\widehat{f(A)}=f(\hat A)$, since $f([x,f(x))_c)=[f(x),f^2(x))_c$ so that:
 $$
   m^u_x(f(A))=m^{cu}_x(\widehat{f(A)})= m^{cu}_x(f(\hat{A}))=D_u m^{cu}_x(\hat A)=D_u m^{u}_x(A),
 $$
Item~\eqref{item-Margulis} is proved, {except for the measurability which will be established at the end of this section.}

\medbreak

We turn to  item \eqref{item-cs-qi}, the quasi-invariance under $cs$-holonomies. {We show the stronger claim in eq.~\eqref{eqLocalCSinv} (recall that a holonomy is by definition a composition of local holonomies).} Since $E^c$ and $E^s$ are transverse, we see that:

\begin{lemma}\label{lem-css}
Let $x,y\in M$ satisfy $\Fcs(x)=\Fcs(y)$ with $d_{cs}(x,y)$ small. Then there exists a {local} $s$-holonomy $h^{s}:U\to V$ with $U$ a neighborhood of $[x,f(x)]_c$ s.t. 
 $$h^{s}([x,f(x))_c)\subset(f^{-1}(y),f^2(y))_c.$$
\end{lemma}

Let us prove the $cs$-quasi-invariance. Since this is a local property, we can assume that $A\subset\Fu_\delta(x)$ for $x\in M^u$ with $\delta>0$ small so that the $cs$-holonomy extends to $h^{cs}(A)=B$. By construction and the Margulis property, 
 $$
   m^{cu}_y(f^{-1}(\hat B)\cup \hat B \cup f(\hat B))=(D_u^{-1}+1+D_u)m^u_y(B).
 $$
By the previous lemma, $h^s(\hat{A})\subset f^{-1}(\hat B)\cup\hat B\cup f(\hat B)$, so the $s$-invariance of the $cu$-conditionals  yields:
 $$
  m^u_x(A) = m^{cu}_x(\hat A) = m^{cu}_y(h^s(\hat A)) \leq (D_u^{-1}+1+D_u)m^u_y(B).
  $$
We have shown $m^{u}_x<<(h^{cs})^{-1}_*(m^{u}_y)$ with Radon-Nikodym derivative bounded by $D_u^{-1}+1+D_u$. Exchanging the roles of $x$ and $y$ we obtain the equivalence of the measures. 

{This finishes the proof of item~\eqref{item-cs-qi} and of Theorem~\ref{mainthm-conditionals} except for the measurability which will follow from the next lemma.}
\end{proof}

{Recall that $M^u=M\setminus\mathcal C^u$ is a measurable subset.} Fix a local product cube $C$ around $p \in {M^u}$.  Using the local product structure, $C$ may be identified as $\mathcal{F}^{cs}_{\loc}(p) \times \mathcal{F}^u_{\loc}(p).$ 

{
\begin{lemma} \label{usc}
For any Borel measurable, bounded function $\phi:M\to\RR$, the map
 $$
    x \in \mathcal{F}^{cs}_{loc}(p) \cap {M^u} \mapsto m^u_x(1_{\Fu_\loc(x)}\cdot\phi)
  $$ 
is measurable.
\end{lemma}

\begin{proof}
Clearly, it is enough to show that for any Borel measurable subset $K \subset C$ the map $$x \in \mathcal{F}^{cs}_{loc}(p) \cap {M^u} \mapsto m^u_x(K_x)$$ is measurable where $K_x = K \cap \mathcal{F}^u_{loc}(x)$. 
{We begin by the special case when $K$ is a {compact} subset.} 
As  $\mathcal{F}^u_{loc}$ varies continuously, and $K$ is {compact,} we conclude that $ x \mapsto K_x$ is upper semicontinuous. Now we proceed to show that  the map $x \mapsto m^u_x(K_x)$ {defined on}  $ \mathcal{F}^{cs}_{loc}(p) \cap \mathcal{C}^u$ is upper semi continuous and in particular measurable. 

As $m^{cu}_x$ is a locally finite measure, given any $\epsilon > 0$ there exists $\epsilon_0>0$ such that $m^{cu}(U_{\epsilon_0} (\hat{K_x})) \leq m^{cu}(\hat{K_x}) + \epsilon$ where $U_{\epsilon_0} (\hat{K_x})$ is the $\epsilon_0-$neighborhood of $K_x$ inside the $cu-$leaf. By construction $m^u_x(K_x) = m^{cu}(\hat{K_x}).$ If $y$ is close enough to $x$  we have 
$$
 h_{y,x}^{s}(\hat{K_y}) \subset U_{\epsilon_0}(\hat{K_x}).
$$
 Indeed, there exists $R > 0$ such that for any such $y \in C,  \hat{K_y} \subset \mathcal{F}^{cu}(y, R).$ Now, if $y$ is close enough to $x$ the restriction of stable holonomy $h_{y, x}^s $ to $\mathcal{F}^{cu}(y, R)$ is close to the identity.
So we obtain that:
\begin{align*}
m^u(K_y) = m^{cu}(\hat{K_y}) &= m^{cu}(h_{y,x}^s(\hat{K_y})) \leq m^{cu}(U_{\epsilon_0}(\hat{K_x}))\\
&\leq m^{cu}(\hat{K_x}) + \epsilon = m^u(K_x) + \epsilon.
\end{align*}  
and this proves the upper semicontinuity in the case $K$ is closed.

Let $\mathcal{K}$ be the collection of all subsets $K \subset C$ such that $x \mapsto m^u_x (K_x)$ is measurable. 
{Note that this is a Dynkin system, i.e.,}
\begin{enumerate}
\item If $A, B \in \mathcal{K}$ and $A \subset B$ then $B \setminus A \in \mathcal{K},$
\item  If $A_1 \subset A_2 \subset \cdots \subset A_n \subset \cdots \in \mathcal{K}$ then $\bigcup A_i \in \mathcal{K}$
\item  $C \in \mathcal{K}.$
\end{enumerate}
{Since $\mathcal K$ contains the closed subsets of $C$ (which form a $\pi-$system), Dynkin $\pi$-$\lambda$-theorem implies that it contains the $\sigma-$algebra generated by the closed subsets, i.e., all the Borel subsets of $C$.}
\end{proof}
}


\section{Dichotomy for flow-type diffeomorphisms}\label{s-dicho}


This section is devoted to the proof of our main result on the MMEs of flow-type diffeomorphisms {(Theorem \ref{mainthm-dichotomy}).}
 
We first build an invariant probability measure $\mu^{cu\otimes s}$  from the $cu$- and $s$-Margulis conditionals. We deduce that dilation factors of these conditionals are inverse of each other: {$D_sD_u=1$}. 
We then show that any MME with nonnegative central exponent has $u$-disintegrations given by the $u$-Margulis conditionals by using the entropy along the unstable foliation as introduced by Ledrappier and Young \cite{LYII} and a classical convexity argument of Ledrappier (see \cite{Ledrappier13} for a pedagogical exposition). 

{A Hopf argument now gives the uniqueness of the MME with nonpositive center exponent. We finish the proof of the dichotomy by building twin measures and considering $f^{-1}$ instead of $f$.}
{We conclude this section by showing some additional properties: uniqueness of the Margulis systems, properties on the sign of the central exponent of $\lambda_c(\mu^{cu\otimes u})$, and the Bernoulli property of hyperbolic MMEs.}

\subsection{Quasi-product measures}
Given a  diffeomorphism $f\in\Diff(M)$ that is flow-type {and minimal strong foliations, we build an invariant probability measure from} the Margulis systems $\{m^{cu}_x\}_{x\in M}$ and $\{m^s\}_{x\in M\setminus\Cyl^s}$  as those provided by Proposition~\ref{cumargulis} and Theorem \ref{mainthm-conditionals} applied to the inverse $f^{-1}$.

\begin{proposition}\label{prop-LPS}
Let $f\in{\Diff^{1}}(M)$ be flow-type with minimal strong foliations.  Assume that there are Margulis systems:
 \begin{itemize}
  \item[--] $\{m^{cu}_x\}_{x\in M}$ on $\Fcu$ with dilation $D^u$ which is $s$-invariant;
  \item[--]  $\{m^{s}\}_{x\in M\setminus\Cyl^s}$ on $\Fs$ with dilation $D^s$ which is $cu$-quasi-invariant;
  \item[--] $\Cyl^s$ satisfies {item~\eqref{item-nucuCsnull}} of Lemma~\ref{lem-CuNull}.
\end{itemize}  
Then there is an invariant Borel probability measure $\mu^{cu \otimes s}$ such that for each $y\in M$, there is a neighborhood in which its conditional measures along $\mathcal{F}^s$ are given by $\{m^s_x\}_{x\in M\setminus\Cyl^s}$ with $m^{cu}_{y}$ as quotient measure. 
Moreover,  $D_u=D_s^{-1}$.
\end{proposition}

\begin{proof}
   We first define local measures $\{m_p\}_{p\in M}$. 
For any $p\in M$ we define a small foliated box (for both $\mathcal{F}^s$ and $\mathcal{F}^{cu}$) neighbourhood $U_p$ containing $p$ as follows: Let $\mathcal{F}^s_{loc}(p),\mathcal{F}^{cu}_{loc}(p)$ be  small  neighborhoods of $p$ in $\mathcal{F}^s(p), \mathcal{F}^{cu}(p)$ and $U_p$ be the homeomorphic image of $\mathcal{F}^s_{loc}(p) \times \mathcal{F}^{cu}_{loc}(p) $  by the local product map $(x,y)\mapsto  \mathcal{F}^{cu}_\loc(x)\cap \mathcal{F}^{s}_\loc(y).$ 

For any $z \in U_p$ by $\mathcal{F}^{s}(z, U_p)$ we mean the plaque of $\mathcal{F}^{s}$ passing through $z$ inside $U_p$.  Similarly we define $\mathcal{F}^{cu}(z, U_p).$ 

One can define two projections $\pi^{cu}_p : U_p \rightarrow \mathcal{F}^{cu}_{loc}(p)$ and $\pi^{s}_{p} : U_p \rightarrow \mathcal{F}^s_{loc}(p)$ in a natural way. By definition, for any $y \in \mathcal{F}^{cu}(p, U_p),  \pi^{s}_p$ sends   $\mathcal{F}^s(y, U_p)$ homeomorphically to $\mathcal{F}^s(p, U_p).$ This is a local $cu-$holonomy map. 

We are going to define a measure $m_p$ on $U_p$ by integrating the measures $m^s_y$ with respect to the measure $m^{cu}_p$. We define
 $$
       m_p := \int_{\mathcal{F}^{cu}(p, U_p)} m^s_y \, dm^{cu}_p(y)
 $$
and justify this definition as follows: Fix $\phi$ a continuous function with $\supp\phi\subset U_p.$  Let $\alpha^p_\phi:\mathcal{F}^{cu}(p, U_p)\to\RR$ be defined by 
 
$$
  \alpha^p_\phi(y):= \left\{ \begin{array}{ll} m_y^s (1_{\Fs (y, U_p)} \cdot \phi)  &\text{ if }y\notin\Cyl^s,\\ 0 & \text{ otherwise.}\end{array}\right.
$$

By item~\eqref{item-nucuCsnull} of Lemma~\ref{lem-CuNull}, the set $\Cyl^s\cap\Fcu (p, U_p)$ has zero $m^{cu}_p$-measure.
By Lemma~\ref{usc},  $\alpha^p_\phi$ is defined and measurable.  
Observe that  $$m^s_y(\Fs(y, U_p))\leq C m^s_p(\Fs(p, U_p))$$ for some constant $C$ given by item (3) of Theorem \ref{mainthm-conditionals} and $m^s_p(\Fs(p, U_p))$ is finite by the Radon property of $m^s_p.$ Therefore  the measurable function $\alpha^p_\phi$ is bounded hence integrable for the Radon measure $m^{cu}_p$.
As announced, we get  a well-defined, finite, positive Borel measures for \emph{every} $p\in M$ by setting:
 \begin{equation}\label{eq-loc-mp}
   m_p (\phi) := \int_{\mathcal{F}^{cu}(p, U_p)} \alpha^p_\phi(y)\, d m^{cu} (y).
 \end{equation}

\begin{claim}
The previously defined measures $m_p$, $p\in M$, with support contained in some neighborhood $U_p$  satisfy the following compatibility condition:
 \begin{equation}\label{eq-mcompat}
   \text{If $p,q\in M$ and $\phi\in C(U_p\cap U_q)$, then $m_p(\phi)=m_q(\phi)$.}
\end{equation}
\end{claim}

To prove this claim, note that, since $U_p$ and $U_q$ are small {(and may be assumed to intersect)}  there exists a well defined $s-$holonomy $h : \pi^{cu}_p(U_p \cap U_q) \rightarrow \pi^{cu}_q(U_p \cap U_p).$

Since $\phi\in C(U_p\cap U_q)$, for any $y \in \pi^{cu}_p(U_p \cap U_q)  $ it follows that
  $$
     \alpha^p_\phi(y)=\alpha^q_\phi(h(y)).
  $$
 Since $\{m^{cu}_x\}_{x\in M}$ is invariant under $s$-holonomy;
 $$\begin{aligned}
     m_p(\phi) &= \int_{\pi^{cu}_p(U_p \cap U_q)} \alpha^p_\phi(y) \, dm^{cu}_p(y)
         = \int_{\pi^{cu}_p(U_p \cap U_q)} \alpha^q_\phi(h(y)) \, dm^{cu}_p(y) \\
        & = \int_{\pi^{cu}_q(U_p \cap U_q)} \alpha^q_\phi(z) \, dm^{cu}_q(z) = m_q(\phi),
   \end{aligned}$$
 proving the claim.
 \medbreak
We now define a finite Borel measure $m$ on $M$ by picking a partition of unity $1=\chi_1+\dots+\chi_r$ subordinated to a finite cover $U_1,\dots,U_r$ determined by points $p_1,\dots,p_r$ and setting:
  $$
     m(\phi) = m_{p_1}(\phi\chi_1)+\dots + m_{p_r}(\phi\chi_r).
  $$
 Finally, we set $\mu^{cu\otimes s} = m(M)^{-1}\cdot m$. Observe it is a Borel probability measure on $M$. It is locally finite, hence finite on the compact set $M$.
 
Observe that $m$ (and $\mu^{cu\otimes s}$) does not depend on the choice of the partition of unity. Indeed, if $1=\chi'_1+\dots+\chi'_{r'}$ is another partition of unity subordinated to some finite cover defined by points $p'_1,\dots,p'_{r'}$, then, for every $\phi\in C(M)$,
 $$\begin{aligned}
   m_{p_1}(\phi\chi_1)+\dots + m_{p_r}(\phi\chi_r) &= \sum_{i=1}^r \sum_{j=1}^{r'} m_{p_i}(\phi\chi_i\chi'_j)
    = \sum_{i=1}^r \sum_{j=1}^{r'} m_{p'_j}(\phi\chi_i\chi'_j) \\
    &= m_{p'_1}(\phi\chi_1)+\dots + m_{p'_r}(\phi\chi_{r'}) .
  \end{aligned}$$
  
The local formula in eq.~\eqref{eq-loc-mp}, valid in any open set $U_p$, implies that the disintegration of $\mu^{cu\otimes s}$ with respect to any partition subordinate to $\Fs$ is given by the Margulis $s$-conditionals $m^s_x$ for a.e. $x\in M$.

To show $f$-invariance, observe that for any  measurable subset $A$ of $U_p\cap f^{-1}(U_q)$ with $p,q\in M$,
 $$\begin{aligned}
   m(f(A)) &= m_q(f(A)) = \int_{\pi^{cu}_q(f(A))} m^s_y(f(A)) \, dm^{cu}_q(y) \\
     &=  D_s \int_{\pi^{cu}_q(f(A))} m^s_{f^{-1}y}(A) \, dm^{cu}_q(y)
     = D_s \int_{\pi^{cu}_p(A)} m^s_x(A) \, d f^{-1}_*(m^{cu}_q)(x) \\
     &= D_sD_u^{-1} m_p(A) = D_sD_u^{-1}m(A).
  \end{aligned}$$
  In the 4th equality above we used that $\pi^{cu}_p(f^{-1}(\pi^{cu}_q(f(A)))) = \pi^{cu}_p(A).$
The same holds for any measurable subset of $M$. The case $A=M$ implies that $D_s=D_u$. Hence $m$ and therefore $\mu^{cu\otimes s}$  are $f$-invariant.
\end{proof}

\subsection{Identification of the conditionals}
We fix some measure of maximal entropy {$\mu$} and we show that its disintegration along $s$- or $u$-conditionals is given by the Margulis systems and its dilation factor is given by the topological entropy.

{
We consider the unstable foliation $\Fu$. We will build an adapted measurable partition $\xi$ satisfying the following properties.
First $\xi$ is subordinated to $\Fu$, i.e., for $\mu$-a.e. $x$,
\begin{itemize}
\item  $\xi(x)$  {is a neighborhood of $x$ in $\cF(x)$ for the intrinsic topology of $\cF(x)$} and with uniformly small diameter,
\item $f(\xi(x))\supset\xi(f(x))$ (i.e., $f\xi$ is less fine than $\xi$, denoted by $f \xi \prec \xi$).
\end{itemize}
Second, $\xi$ is generating, that is,
\begin{itemize}
\item $\bigvee_{i\geq0} f^{-i} \xi$ is the partition into points {(modulo $\mu$).}
\end{itemize}

To find such a partition $\xi$ we use the following construction from \cite[Proposition 3.1 and Lemma 3.2]{YangYang}.
For any finite partition $\mathcal{P}$ 
we denote by $\mathcal{P}^u$ the finer partition obtained by intersecting local plaques of $\mathcal{F}^u$ with elements of $\mathcal{P}.$

\begin{proposition} \label{subordinatedY}
There exists  a finite measurable partition $\mathcal{P} = \{P_i\}$ such that:
\begin{itemize}
 \item  $\mu(\bigcup_i {\rm int}(P_i)) =1$;
 \item each $P_i$ is a contained in a foliation chart of $\mathcal{F}^u$;
 \item  $\bigvee_{i=0}^{\infty} f^i(\mathcal{P}^u)$ is a  measurable generating partition subordinated to $\mathcal{F}^u.$
\end{itemize}  
\end{proposition}

Apply this proposition to get $\mathcal P$ and $\mathcal P^u$ and set  $\xi:=\bigvee_{i=0}^{\infty} f^i(\mathcal{P}^u)$.

Let $\{\mu_{\xi(x)}\}_{x\in M}$ be the  disintegration of $\mu$ w.r.t. $\xi$. The entropy with respect to $\cF$ is defined  \cite{LY85b} as:
 $$
    h(f,\mu,\cF) = - \int \log\mu_{\xi(x)}(f^{-1}\xi(fx))\, d\mu .
 $$
Note that one could replace $\xi$ in the above formula by any generating, measurable partition  subordinate to $\Fu$ \cite{LY85b}.  

  The next result follows from Theorem $C^{'}$ in \cite{LYII} and items (i)-(iii) after the statement of the theorem.}

\begin{proposition}[Ledrappier-Young]
Let $ f \in\Diff^2(M)$ be partially hyperbolic with {strong} unstable foliation $\Fu$. For any  $\mu\in\Proberg(f)$,
 $$
    h(f,\mu,\cF^u) \leq h(f,\mu).
 $$
If all center Lyapunov exponents of $\mu$ are nonpositive, then  the above inequality is an equality.
\end{proposition}
Indeed, as $\mu$ has non-positive center exponents there exists $\lambda > 0$ such that the corresponding measurable foliation 
$$
  {\mathcal W^u(x) :=} \left\{ y \in M : \limsup_{n \rightarrow \infty} \frac{1}{n}   \log d(f^{-n}(x), f^{-n}(y)) < -\lambda\right\}.
$$
coincides a.e.\ with $\Fu$.

Using the above result  we are able to show the following proposition which is an extension of Ledrappier's argument in \cite{L84} (see also \cite{Ledrappier13}).   {The next two results prove Theorem \ref{mainthm-mme} and complete the proof of Theorem \ref{mainthm-conditionals}.}

\begin{proposition}\label{prop-Ledrappier}
Let $f\in\Diff^2(M)$ be a  diffeomorphism with a Margulis $u$-system $\{m^u_x\}_{x\in M\setminus\Cyl^u}$  with dilation factor $D_u>1$ and $\mu^u_x$ fully supported on $\mathcal{F}^u(x)$ for each $x\in M\setminus\Cyl^u$. If $\mu\in\Proberg(f)$ with $\lambda^c(\mu)\leq 0$, then $h(f,\mu)\leq\log D_u$. Moreover, there is equality if and only if the disintegration of $\mu$ along $\Fu$ is given by $m^u_x$, $\mu$-a.e.
\end{proposition}

\begin{proof}
First suppose that $\mu$ gives full measure to $\mathcal{C}^u$. {Item~\eqref{item-muCsnull} of Lemma~\ref{lem-CuNull} implies that $h(f, \mu)=0$ so $h(f, \mu)< \log D_u$.  }
Now suppose that $\mathcal{C}^u$ has zero $\mu$-measure. {Following Ledrappier, we will compare the conditional measures $\mu_{\xi(x)}(f^{-1}\xi(fx))$ that appear in the unstable entropy with similarly defined quantities from the unstable Margulis system.}

Define a normalized family of measures adapted to the partition $\xi$, for $\mu$-a.e. $x\in M$,
$$
 m_x (A) := \frac{m_x^u (A \cap \xi(x))}{ m^u_x(\xi(x))}.
$$ Observe that the above ratio is well defined, as $m_x^u$ is fully supported and $\xi(x)$ is a neighbordhood of $x$.
The dilation property of Margulis measures yields: 
$$
m_x((f^{-1}\xi)(x)) = D_u^{-1} \frac{m^u_{f(x)} (\xi( fx)) }{m^u_x(\xi(x))}.
$$ 
This can be expressed in terms of the function
 $$
   g(x):=-\log m_x((f^{-1}\xi)(x)) \geq0.
 $$ 

{
\begin{lemma}
 The map
$$ x \mapsto m^u_x(\xi(x))$$ is measurable.
 Consequently the above function $g$ 
 is measurable. 
 \end{lemma}
\begin{proof}
A partition by foliated boxes is a finite partition $\mathcal P$ modulo $\mu$ such that each $A\in\mathcal P$ has the form $(-1,+1)^d$ for some $u$-foliation chart $\chi_A$.

We work with a finer partition $\mathcal P^u$ which is measurable  and  obtained from foliated box partition $\mathcal P$, where
 $$
     \mathcal P^u :=\{ \chi_A((-1,1)^{d_u}\times\{t\}):t\in(-1,+1)^{d-d_u},{A\in\mathcal P}\}.
   $$

For $n \in \mathbb{N}$ define
 $$
    \mathcal{P}^u_n := \bigvee_{1\leq k\leq n} f^k(\mathcal{P}^u) \text{ and }
    \mathcal{P}_n := \bigvee_{1\leq k\leq n} f^k(\mathcal{P}).
 $$ 
Note that $\mathcal{P}^u_n  \nearrow \xi.$

Let $m^{u, \mathcal{P}}_x$ be the restriction (not normalized) of $m^u_x$ to the atom $\mathcal{P}^u(x)$ of the partition $\mathcal{P}^u.$ Observe that $m^{u, \mathcal{P}}_x$ may be seen as a measure defined on $M.$ Since $\mathcal P^u(x)$ is bounded hence relatively compact in the intrinsic topology of $\Fu(x)$, its $m_x^u$-measure is finite.
As $\xi(x) \subset \mathcal{P}^u(x)$ and $\mathcal{P}^u_n(x) \searrow \xi(x)$, we get: $$ m^u_x (\xi(x)) = m^{u, \mathcal{P}}_x(\xi(x)) = \lim_{n \rightarrow \infty} m^{u, \mathcal{P}}_x (\mathcal{P}^u_n (x)). $$

To conclude, it is enough to see that for each $n\geq1$,
 the map $x \mapsto m^{u, \mathcal{P}}_x (\mathcal{P}^u_n (x))$ is measurable. But Lemma \ref{usc} shows that $x\mapsto m^{u, \mathcal{P}}_x (A)$ is measurable for each element $A$ of the finite partition $\mathcal P_n$.
\end{proof} 
}

By the pointwise ergodic theorem, we know $\lim_n \frac1nS_ng(x)$ (possibly $+\infty$) exists almost everywhere.  To identify this limit,  observe that it is also a limit in probability. Taking the logarithm of the previous identity, we see that $g(x)=h(fx)-h(x)+\log D_u$ for a measurable function $h$. Therefore the limit in probability and therefore almost everywhere is the constant  $\log D_u$. Thus $g$ is integrable with:
 $$
- \int \log m_x((f^{-1} \xi)( x)) d \mu =  \log D_u.
 $$
Now recall that $ h(f,\mu,\mathcal{F}^u) = - \displaystyle\int \log\mu_{\xi(x)}((f^{-1}\xi) (x))\, d\mu $ and so,
$$
- \int \log\mu_{\xi(x)}((f^{-1}\xi)(x))\, d\mu \leq - \int \log m_x((f^{-1}\xi)(x)) d\mu = \log(D_u).
$$

The inequality comes from {Jensen's inequality and the (strict) concavity of the logarithm.
The case of equality for Jensen's inequality yields that this is an equality if and only if $\mu_{\xi(x)}((f^{-1}\xi)(x)) = m_x((f^{-1}\xi)(x))$ for $\mu$-a.e. $x\in M$. Replacing $\xi$ by $f^{-n}\xi$, we obtain that
 $$
      \mu_{f^{-n}\xi(x)}(f^{-n-1}\xi)(x)) = \frac{m^u_x((f^{-n-1}\xi)(x))}{m^u_x((f^{-n}\xi)(x))}
  $$
so
  $$\begin{aligned}
    \mu_{\xi(x)}((f^{-n-1}\xi)(x)) &= \prod_{k=0}^n \frac{\mu_{\xi(x)}((f^{-k-1}\xi)(x))}{\mu_{\xi(x)}((f^{-k}\xi)(x))}
    =
    \prod_{k=0}^n \mu_{(f^{-k}\xi)(x)}((f^{-k-1}\xi)(x))\\
     &
     =\prod_{k=0}^n \frac{m^u_x((f^{-k-1}\xi)(x))}{m^u_x((f^{-k}\xi)(x))}    
    =\frac{m^u_x((f^{-n-1}\xi)(x))}{m^u_x(\xi(x))}.
  \end{aligned}$$
Since $\xi$ is generating and increasing,  the disintegration of $\mu$ along $\xi$ is given by the Margulis $u$-conditionals as claimed.}
\end{proof}

The next result identifies the dilation factor {and proves the Addendum~\ref{addendum}.}

{\begin{proposition}\label{prop:dilationentropy}
Let $f\in\Diff^2(M)$ be a flow-type diffeomorphism  with minimal strong foliations
and $D_u=D^{-1}_s$ as in Proposition \ref{prop-LPS}, then $D_u = \exp(h_{top}(\phi)). $
\end{proposition}}

\begin{proof}
Let $\mu$ be an ergodic measure for $f$.  If $\lambda^c(\mu)\leq 0$.  Then by Proposition \ref{prop-Ledrappier} we know that $h(f, \mu)\leq \log D_u$.  On the other hand, if $\lambda^c(\mu)>0$, then for $f^{-1}$ we see that {$h(f, \mu)=h(f^{-1}, \mu)\leq \log D_s^{-1}=\log D_u$.}  Hence, for any ergodic measure we have $h(f, \mu)\leq \log D_u$ and by the Variational Principle we have $h_\mathrm\top(f)\leq \log D_u.$

To finish the proof, it is enough to show that, $\mu^{cu \otimes s} $ being the measure constructed in Proposition \ref{prop-LPS}, 
$$h(\mu^{cu \otimes s} ) \geq \log(D^u).$$  
Since the conditional measures of  $\mu^{cu \otimes s}$ along $\mathcal{F}^s$ are given by the Margulis system $\{m^s\}_{x\in M\setminus\Cyl^s}$, the previous proposition shows that:
$$
 h(\mu^{cu \otimes s} , f^{-1}, \mathcal{F}^s) = - \log(D_s).
$$
Note that $\mathcal{F}^s$ is the strong unstable foliation of $f^{-1}.$ By Proposition~\ref{prop-Ledrappier},
 $$
  h(\mu^{cu \otimes s} , f) = h(\mu^{cu \otimes s} , f^{-1}) \geq -\log(D_s) = \log(D_u).
 $$
\end{proof}

{
\begin{proof}[Proof of Theorem~\ref{mainthm-mme}]
Let $f$ be a $C^2$ diffeomorphism of a compact manifold $M$. Assume that $f$ is a flow-type diffeomorphism and minimal strong foliations. The dilation factor $D_u$ is greater than $1$ by Lemma~\ref{lem-Dbigger}. By Proposition~\ref{prop:dilationentropy}, $D_u=D_s^{-1}=\exp h_\top(f)$.

Therefore Propostion~\ref{prop-Ledrappier} shows that any ergodic MME $\mu$ with a nonpositive central exponent has unstable conditionals given by the unstable Margulis system $\{m^u_x\}_{x\in M ^u}$.

{Finally we prove that $\mu$ has full support. Observe that by Proposition  \ref{prop-Ledrappier}  the conditional measures of $\mu$ with respect to a subordinated partition $\xi$ are equal to (normalized restrictions of the) Margulis measures to the atom of $\xi$. So, there exists an atom of partition such that almost every (with respect to normalized Margulis measure) point  is contained in the support of $\mu.$ By construction, each atom of $\xi$ contains an open set inside unstable leaf. As Margulis measures are fully supported we get an open set (inside unstable leaf) totally contained in the support of $\mu$. By invariance of  support and minimality of unstable foliation we get $\supp(\mu)=M.$}
\end{proof}
}

{
\subsection{Hyperbolic MMEs}
In this section we assume the existence of some hyperbolic MME. We build its twin measure, which is a MME with opposite central exponent. We get a uniqueness result from a version of Hopf's argument.
}

\begin{proof}[Proof of Proposition \ref{mainthm-twin}]
By assumption (1), the exponent along $\Fc(x)$ is negative for almost every $x$. {
Let $\mathcal W^c(x):=\{y\in\Fc(x):\lim_{n\to\infty} d(f^nx,f^ny)=0\}$. Since the central foliation is one-dimensional, it is not difficult to see that $\mathcal W^c(x)$ is an open arc that contains, for $\mu$-a.e. $x\in M$, a neighborhood of $x$ in its central leaf. Moreover,  the $C^1$ submanifold $\mathcal W^c(x)$ depends measurably on $x$. These two facts are immediate consequences of 
 \cite[Theorem 3.11]{ABC}  which deals more generally with stable manifolds in the $C^1$ setting with dominated splitting).}
By item~(2), this curve is bounded. Thus, the following is well-defined $\mu$-almost everywhere (recall that $\Fc(x)$ is oriented):
 $$
    \beta: M \to M,\quad x\longmapsto \sup \mathcal W^c(x).
 $$
Note that {$\beta(x)$ must be one of the endpoints of $W^c(x)$. Moreover,} $\beta$ is measurable  and satisfies $\beta\circ f=f\circ\beta$ and $\beta(x)\in\Fc(x)$ for all $x\in M$.

\begin{claim}\label{cla.twin}
There is a measurable subset $Z\subset M$ with $\mu(Z)=1$ such that the restriction $\beta|Z$ is injective. 
\end{claim}

We defer the proof of the claim and first deduce the proposition from it.   Let $\nu:=\beta_*(\mu)$. Note that the claim  implies that $\beta$ is a measure-preserving conjugacy between $(f,\mu)$ and $(f,\nu)$. If $\lambda_c(\nu)$ was negative, {for $\mu$-a.e. $x\in M$, $\beta(x)$  would be contained in the open $\mathcal W^c(\beta(x))$  contradicting that $\beta(x)\in\partial\mathcal W_c(x)$.} Thus $\lambda_c(\nu)\geq0$ and the proposition is established. \end{proof}

To prepare the proof of the claim, recall from Section~\ref{s-disint} that one can find a measurable disintegration  of $\mu$ along the foliation $\Fc$ into  Radon measures $\{\bar\mu^c_x\}_{x\in M}$ (well-defined up to a  factor, possibly depending on the leaf). {Note that $\mu$ is atomless since otherwise it would be carried by a periodic orbit hence by compact center leaves, in contradiction to assumption (3). Therefore} $f(x)\in\Fc(x)$ with $x\ne f(x)$ for $\mu$-a.e. $x$ and one can define Radon measures $\{\mu^c_x\}_{x\in M}$  by using the invariant normalization $\mu^c_x([y,f(y))_c)=1$ for $\mu$-a.e.\ $x\in M$ and for all $y\in\Fc(x)$. We first show:
 \begin{equation}\label{eq-concentrate}
   \text{for $\mu$-a.e. }x\in M\quad \mu^c_x|\mathcal W^c(x)=c_x\delta_{x} \quad
   \text{for some }c_x>0.
 \end{equation}

For each $\eps>0$, let {$\widehat{\mathcal W}^c_\eps(x):=\{y\in\mathcal W^c(x):d_c(y,\partial\mathcal W^c(x))>\eps\}$ where $d_c(\cdot, \cdot)$ is the induced distance on the center foliation}. 
Note that for any $\delta>0$, there are a set $S$ of positive $\mu$-measure of points $x$ and some arbitrarily small $\eps>0$ such that 
 $$
  \forall x\in S\quad  \mu^c_x(\widehat{\mathcal W}^c_\eps(x))\geq (1-\delta)\mu^c_x(\mathcal W^c(x)).
  $$   
The last measure is positive for $\mu$-a.e.\ $x$. Now, for $\mu$-a.e.\ $x$, for all large $n\geq0$, 
 $$
   f^n(\widehat{\mathcal W}^c_\eps(f^{-n}x))\subset \{y\in \mathcal W^c(x): d_c(y,x)<\eps\}.
 $$
The ergodicity of $\mu$ implies that, for $\mu$-a.e. $x$, $f^{-n}x\in S$ for some arbitrarily large integers $n$. Hence, by invariance of $\mu$:
 $$\begin{aligned}
    \mu^c_x\left(\{y\in \mathcal W^c(x): d_c(y,x)<\eps\}\right) 
       &\geq \mu^c_{f^{-n}x}(\widehat{\mathcal W}^c_\eps(f^{-n}x)) \\
       &\geq (1-\delta) \mu^c_{f^{-n}x}(\mathcal W^c(f^{-n}x))
    = (1-\delta) \mu^c_{x}(\mathcal W^c(x)).
   \end{aligned}$$
 Since $\eps,\delta>0$ were arbitrarily small, eq.~\eqref{eq-concentrate} follows.
 
 \smallbreak
 
\begin{proof}[Proof of Claim \ref{cla.twin}] To prove the claim, let $Z$ be the set of $x\in M$ for which  $\mu^c_x$ satisfies eq.~\eqref{eq-concentrate}. This is a measurable set with full measure $\mu$. Now, let $x,y\in Z$.  If $\beta(x)=\beta(y)$ then $\mathcal W^c(x)=\mathcal W^c(y)$. In particular $\mu^c_x|\mathcal W^c(x)=\mu^c_y|\mathcal W^c(y)$. By eq.~\eqref{eq-concentrate}, this implies $x=y$, concluding the proof of the claim. 
\end{proof}
 
{We now turn to the Hopf argument.}

\begin{proof}[Proof of Proposition \ref{mainthm-hopf}]
{Let $\mu,\nu$ be ergodic MME with $\lambda_c(\mu)<0$ and $\lambda_c(\nu)\leq 0$.} 
Proposition \ref{prop-Ledrappier} implies that their conditional measures along unstable foliation  {are both given by the $u$-Margulis system $\{m^u_x\}_{x\in M\setminus\Cyl^u}$.} 
  Let {$B_{\mu}:=\{x\in M: \tfrac1n\sum_{k=0}^{n-1}\delta_{f^kx}\rightharpoonup\mu\}$} be the ergodic basin of $\mu$ (the convergence is in the weak star topology as $n\to+\infty$). By ergodicity $\mu(M\setminus B_{\mu}) =0$  and so $m_x^u (M\setminus B_{\mu}) =0$ for $\mu-$a.e $x$. Similarly, letting $B_\nu$ be the ergodic basin of $\nu$, 
  for $\nu-$a.e $y$ we have $m^u_y(M\setminus B_{\nu}) =0.$   

 As the center Lyapunov exponent of $\mu$ is negative, {for $\mu$-a.e. $x \in B_{\mu}$, $m^{u}_x(M\setminus  B_{\mu})=0$ and there is  a subset $ K \subset \mathcal{F}^u(x)\cap B_\mu$ with  $m^u_x(K) >0$ and such that the size of the Pesin local stable manifolds $\mathcal W^s_{loc}(z)$ of points $z\in K$ is uniformly bounded from below.} 

Pick $y \in B_{\nu}$ with $m^u_y(M\setminus B_{\nu})=0$. The density of $\mathcal{F}^u(y)$ implies that {there is a local $cs$-holonomy $h:K\to\Fu(x)$  (perhaps after replacing $K$ with a smaller subset). This holonomy is absolutely continuous from $(K,m^u_x)$ to $(\Fu(x),m^u_y)$, hence:
$$
 m^u_y\left(h(K)\right) > 0.
$$ 
{By the choice of $K$, $h(z)$ belongs to the Pesin stable manifold of $z$.}
Since the ergodic basin is saturated by stable manifolds, $h(K)\subset  B_{\mu}$} and therefore  $m^u_y(B_{\mu})>0.$ As  $m^u_y(M\setminus B_{\nu}) = 0$,  we conclude that $B_{\nu} \cap B_{\mu} \neq \emptyset$ and consequently $\mu=\nu.$
\end{proof}

\begin{remark}\label{rem-unique}
Let $f$ be a flow-type $C^2$ diffeomorphism with minimal strong foliations. Whenever $\mu$ is an ergodic MME with $\lambda_c(\mu)<0$, it is the only ergodic invariant probability measure whose $u$-conditionals are given by an unstable Margulis system.
\end{remark}

\begin{proof}[Proof of Theorem \ref{mainthm-dichotomy}]
{Let $f$ be a  flow-type $C^2$ diffeomorphism and minimal strong foliations. Observe that since $f$ is partially hyperbolic with a one-dimensional center, there must exist an ergodic MME $\mu$ \cite{CowiesonYoung,DFPV, LVY}.}  For instance,  assume that $\lambda^c(\mu)<0$ (if not,  use $f^{-1}$).  

Proposition \ref{mainthm-hopf} shows that there is no other ergodic MME with nonpositive central exponent. Now Proposition \ref{mainthm-twin} gives some ergodic MME $\nu$ with $\lambda^c(\nu)\geq 0$. In particular $\nu\ne\mu$.  The previous uniqueness shows that the case $\lambda^c(\nu)=0$ cannot occur so $\lambda^c(\nu)>0$. We see that there are exactly two ergodic MME. {The dichotomy is proved.

\medbreak

It remains to prove that $\mu$ is Bernoulli.}
The symbolic dynamics of \cite{Benovadia} implies that $(f,\mu)$ is isomorphic to the product of a Bernoulli measure with a circular permutation of some order $p\geq1$. It follows that $\mu=(\mu_1+\dots+\mu_p)/p$ where the measures $\mu_1,\dots,\mu_p$ are distinct  ergodic MMEs for $f^p$ such that $\lambda^c(f^p,\mu_k)=p\cdot\lambda^c(f,\mu)$ {does not depend on $k$ and is not zero.} Observe that $f^p$ like $f$ is a flow-type diffeomorphism with minimal strong foliations. Hence, the previous uniqueness result applies to $f^p$ showing that $p=1$, i.e., $(f,\mu)$ is Bernoulli.

The theorem is established.
\end{proof}

\section{Perturbing time-one maps to get flow-type with minimal strong foliations} \label{s-pertub-flow}

{We prove Theorem \ref{mainthm-mft}. Let $\varphi^t:M\to M$ be a topologically transitive Anosov flow on a compact manifold and let $T>0$.  We find an open set of diffeomorphisms that are  flow-type and minimal strong foliations accumulating on $\varphi^T$. We first use structural stability results, mainly from \cite{HPS}, to see that {all $C^1$ perturbations of $\varphi^T$ are flow-type. A $C^1$ perturbation will then be used to get the robust} minimality of the strong foliations. 
}

\begin{proposition}\label{prop-mft}
Let $\varphi^T:M\to M$ be the time $T>0$ map of a 
Anosov flow. Then there is a $C^1$-neighborhood $\mathcal V$ of $\phi^T$ such that $f\in\mathcal V$ is flow-type.
\end{proposition}

\begin{proof}
We must prove \mftPH, \mftDC, \mftCLeaf, and \mftPush\ for all diffeomorphisms $C^1$-close to $\varphi^T$. Observe that these properties are well-known for $\varphi^T$ itself. Let us see that they hold for all $C^1$-close diffeomorphisms using the structural stability theory in \cite{HPS}.

Partial hyperbolicity with the center subbundle of a given dimension is well-known to be robust.
Since the center foliation $\Fc$ of $\varphi^T$ is the partition into the orbits of the flows, {it} is smooth, {hence} it is plaque expansive \cite[(7.2)]{HPS}. Therefore $(\varphi^T,\Fc)$ is structurally stable  \cite[(7.1)]{HPS}, {and} dynamical coherence \mftDC\ holds robustly. Indeed, this theorem yields a center foliation $\Fc_g$ for $g$, whereas its proof, especially \cite[thm. (6.8)]{HPS}, gives stable and unstable manifolds of center leaves, {these provide the required invariant foliations $\Fcu_g$ and $\Fcs_g$ tangent to $E^{cu}$ and $E^{cs}$.} 

The center foliation of $\varphi^T$ is obviously oriented by the vector field. The flow being expansive it has at most countably many closed orbits. {The Anosov Closing Lemma yields} (infinitely many) closed orbits. The structural stability of $(\varphi^T,\Fc)$ implies that \mftCLeaf\ holds robustly.

To establish \mftPush, we need to refer to the proof of the structural stability of $(\varphi^T,\Fc)$ \cite[Thm. (7.1)]{HPS} and especially of \cite[Thm. (6.8)]{HPS}. We use terminology and notations from \cite[chap. 6, 7]{HPS}. On page 107 of \cite{HPS}, it is shown that the perturbed diffeomorphism $f'$ has a center foliation $\cF'$ whose every leaf is close to the corresponding leaf of $\cF$ in the sense that $\cF'(h(x))$ is represented by section of the formal normal bundle to $\cF(x)$ which is close to zero when $f'$ is close to $f=\varphi^T$. Since the lifts $i^*f,i^*f'$ of $f,f'$ to this formal bundle are close to each other, we see that $f'(x)=F'(\tau(x),x)$ for some function $\tau:M\to\RR$ with $\sup|\tau-T|<\eps$ where $\eps>0$ can be taken arbitrarily small by assuming $f'$ to be close enough to $f=\varphi^T$.

The previous reasoning shows that the minimum of the lengths of the closed loops for $\cF'$ is close to that of $\cF$. In particular it can be assumed to be larger than $3\eps$.

If $\tau$ is not continuous, then there are two sequences of points $x_n,y_n\in M$ that converge to the same limit $z$ and such that $s:=\lim_{n\to\infty} \tau(x_n)$ and $t:=\lim_{n\to\infty}\tau(y_n)$ exist and are distinct. This implies that $F^s(z)=F^t(z)$ with $|t-s|<2\eps$, i.e., $\cF'$ contains a loop of length less than $2\eps$. The contradiction proves the continuity of $\tau$.
\end{proof}

{The robust minimality of the strong foliations follows from works of Bonatti, Di\'az and Ur\`es  \cite{BD96,BDU}.} Let $\mathcal T$ be the set of time $T$ maps of  topologically transitive Anosov flows on some compact manifold. We need to find an open set $\mathcal U$ of diffeomorphisms  whose strong stable and strong unstable foliations are both minimal and such that $\mathcal T\subset\overline{\mathcal U}$. {We proceed in two steps.}

\newcommand\step[2]{\medbreak\noindent{\bf Step #1.} \emph{ #2}\medbreak}

\step{1}
{Any diffeomorphism in $\mathcal T$ can be $C^1-$approximated by  robustly transitive diffeomorphisms.}

{This follows from the work of Bonatti and Diaz in \cite{BD96}.}

\step{2}
{Any diffeomorphism in $\mathcal T$  has a $C^1-$neighborhood $\mathcal V$ with the following property. Any  robustly transitive diffeomorphism in $\mathcal V$ can be $C^1$-approximated by diffeomorphisms whose strong stable and strong unstable foliations are both minimal.
}

{
This follows from the work of Bonatti, D\'iaz and Ures \cite{BDU} (it is in fact a simpler situation since in our setting hyperbolic periodic points are dense and contained in invariant compact leaves). 
}

\medbreak

This completes the proof of Theorem \ref{mainthm-mft}.

{
  \begin{remark}
 A recent work of Ures, Viana, and Yang \cite{UresComm} extends  \cite{BDU} to a larger class of systems among the $C^\infty$ volume preserving diffeomorphisms. More precisely, they prove that in a neighborhood of the time one map of the geodesic flow of any hyperbolic surface, there is an open and dense set of diffeomorphisms with both strong foliations minimal.
 \end{remark}
 }

\bibliography{ourbib-BFT2016}{}
\bibliographystyle{plain}
\end{document}